\documentclass{article}

\AtEndDocument{\bigskip{\footnotesize%
  \textsc{University of Luxembourg, Mathematics Research Unit, 6 rue Richard Coudenhove-Kalergi, L-1359 Luxembourg.} \par
  \textit{E-mail address:} \texttt{francois.petit@uni.lu}
  }}

\usepackage[francais,english]{babel}
\usepackage[T1]{fontenc}
\usepackage[latin1]{inputenc}
\usepackage{amsmath,amssymb,amsthm,mathrsfs}
\usepackage[all,2cell]{xy}
\usepackage{graphicx}
\usepackage{enumerate} 
\usepackage{lmodern}
\usepackage{mathrsfs} 
\usepackage{dsfont}
\usepackage{color}

\DeclareMathOperator{\fRHom}{R\mathcal{H}om}

\DeclareMathOperator{\Hom}{Hom}

\DeclareMathOperator{\id}{id}
\DeclareMathOperator{\opp}{op}
\DeclareMathOperator{\Mod}{Mod}

\DeclareMathOperator{\Hn}{H}

\DeclareMathOperator{\fHom}{\mathcal{H}om}

\DeclareMathOperator{\gr}{gr}
\DeclareMathOperator{\Rg}{R\Gamma}

\DeclareMathOperator{\Supp}{Supp}

\begin{document}

\theoremstyle{plain} % style plain
\newtheorem{theorem}{Theorem}[section]
\newtheorem{corollary}[theorem]{Corollary}
\newtheorem{proposition}[theorem]{Proposition}
\newtheorem{lemma}[theorem]{Lemma}
\theoremstyle{definition} % style definition
\newtheorem{definition}[theorem]{Definition}
\newtheorem{example}[theorem]{Example}
\newtheorem{examples}[theorem]{Examples}
\newtheorem{remark}[theorem]{Remark}
\newtheorem{notation}[theorem]{Notations}

\numberwithin{equation}{section}

\newcommand{\On}[1]{\mathcal{O}_{#1}}
\newcommand{\En}[1]{\mathcal{E}_{#1}}
\newcommand{\Fn}[1]{\mathcal{F}_{#1}} 
\newcommand{\tFn}[1]{\mathcal{\tilde{F}}_{#1}}
\newcommand{\hum}[1]{hom_{\mathcal{A}}({#1})}
\newcommand{\hcl}[2]{#1_0 \lbrack #1_1|#1_2|\ldots|#1_{#2} \rbrack}
\newcommand{\hclp}[3]{#1_0 \lbrack #1_1|#1_2|\ldots|#3|\ldots|#1_{#2} \rbrack}
\newcommand{\catMod}{\mathsf{Mod}}
\newcommand{\Der}{\mathsf{D}}
\newcommand{\Ds}{D_{\mathbb{C}}}
\newcommand{\DG}{\mathsf{D}^{b}_{dg,\mathbb{R}-\mathsf{C}}(\mathbb{C}_X)}
\newcommand{\lI}{[\mspace{-1.5 mu} [}
\newcommand{\rI}{] \mspace{-1.5 mu} ]}
\newcommand{\Ku}[2]{\mathfrak{K}_{#1,#2}}
\newcommand{\iKu}[2]{\mathfrak{K^{-1}}_{#1,#2}}
\newcommand{\Be}{B^{e}}
\newcommand{\op}[1]{#1^{\opp}}
\newcommand{\N}{\mathbb{N}}
\newcommand{\Ab}[1]{#1/\lbrack #1 , #1 \rbrack}
\newcommand{\Du}{\mathbb{D}}
\newcommand{\C}{\mathbb{C}}
\newcommand{\Z}{\mathbb{Z}}
\newcommand{\w}{\omega}
\newcommand{\K}{\mathcal{K}}
\newcommand{\Hoc}{\mathcal{H}\mathcal{H}}
\newcommand{\env}[1]{{\vphantom{#1}}^{e}{#1}}
\newcommand{\eA}{{}^eA}
\newcommand{\eB}{{}^eB}
\newcommand{\eC}{{}^eC}
\newcommand{\cA}{\mathcal{A}} 
\newcommand{\cB}{\mathcal{B}}
\newcommand{\cD}{\mathcal{D}}
\newcommand{\cR}{\mathcal{R}}
\newcommand{\cI}{\mathcal{I}}
\newcommand{\cL}{\mathcal{L}}
\newcommand{\cO}{\mathcal{O}}
\newcommand{\cM}{\mathcal{M}}
\newcommand{\cN}{\mathcal{N}}
\newcommand{\cK}{\mathcal{K}}
\newcommand{\cC}{\mathcal{C}}
\newcommand{\cF}{\mathcal{F}}
\newcommand{\cG}{\mathcal{G}}
\newcommand{\cP}{\mathcal{P}}
\newcommand{\cQ}{\mathcal{Q}}
\newcommand{\cU}{\mathcal{U}}
\newcommand{\cE}{\mathcal{E}}
\newcommand{\cS}{\mathcal{S}}
\newcommand{\cT}{\mathcal{T}}
\newcommand{\chE}{\widehat{\mathcal{E}}}
\newcommand{\cW}{\mathcal{W}}
\newcommand{\chW}{\widehat{\mathcal{W}}}
\newcommand{\Hper}{\Hn^0_{\textrm{per}}}
\newcommand{\Dper}{\Der_{\mathrm{perf}}}
\newcommand{\Yo}{\textrm{Y}}
\newcommand{\gqcoh}{\mathrm{gqcoh}}
\newcommand{\coh}{\mathrm{coh}}
\newcommand{\cc}{\mathrm{cc}}
\newcommand{\qcc}{\mathrm{qcc}}
\newcommand{\gd}{\mathrm{gd}}
\newcommand{\qcoh}{\mathrm{qcoh}}
\newcommand{\lcl}{\mathrm{lcl}}
\newcommand{\fin}{\mathrm{fin}}
\newcommand{\obplus}[1][i \in I]{\underset{#1}{\overline{\bigoplus}}}
\newcommand{\Lte}{\mathop{\otimes}\limits^{\rm L}}
\newcommand{\te}{\mathop{\otimes}\limits^{}}
\newcommand{\pt}{\textnormal{pt}}
\newcommand{\A}[1][X]{\cA_{{#1}}}
\newcommand{\dA}[1][X]{\cC_{X_{#1}}}
\newcommand{\conv}[1][]{\mathop{\circ}\limits_{#1}}
\newcommand{\sconv}[1][]{\mathop{\ast}\limits_{#1}}
\newcommand{\reim}[1]{\textnormal{R}{#1}_!}
\newcommand{\roim}[1]{\textnormal{R}{#1}_\ast}
\newcommand{\ldetens}{\overset{\mathnormal{L}}{\underline{\boxtimes}}}
\newcommand{\br}{\bigr)}
\newcommand{\bl}{\bigl(}
\newcommand{\sC}{\mathscr{C}}
\newcommand{\ucat}{\mathbf{1}}
\newcommand{\ubtimes}{\underline{\boxtimes}}
\newcommand{\uLte}{\mathop{\underline{\otimes}}\limits^{\rm L}} 
\newcommand{\Lp}{\mathrm{L}p}
\newcommand{\pder}[3][]{\frac{\partial^{#1}#2}{\partial{#3}}}
\newcommand{\reg}{\mathrm{reg}}
\newcommand{\sing}{\mathrm{sing}}
\newcommand{\fExt}{\mathcal{E}xt}
\newcommand{\fTor}{\mathcal{T}or}
\newcommand{\fEnd}{\mathcal{E}nd}
\newcommand{\dL}{\mathrm{L}}
\newcommand{\fgd}{\mathrm{fgd}}
\newcommand{\Zl}{Z}

\title{Quantization of spectral curves and DQ-modules}
\author{Fran\c{c}ois Petit\footnote{The author has been fully supported in the frame of the OPEN scheme of the Fonds National de la Recherche (FNR) with the project QUANTMOD O13/570706 \newline \noindent \textit{2010 Mathematics Subject Classification.} 53D55, 32C38, 14D21, 14F10.}}
\date{}

\maketitle

\begin{abstract} 
Given an holomorphic Higgs bundle on a compact Riemann surface of genus greater than one, we prove the existence of an holonomic DQ-module supported by the spectral curve associated to this bundle. Then, we relate quantum curves arising in various situations (quantization of spectral curves of Higgs Bundles, quantization of the $A$-polynomial...) to DQ-modules and show that a quantum curve and the DQ-module canonically associated to it have isomorphic sheaves of solutions.
\end{abstract}
\setcounter{tocdepth}{2}
\tableofcontents
\section{Introduction}
Spectral curves arose first in the study of certain integrable systems as the zero locus of families of characteristic polynomials. The notion of spectral curves has nowadays a broader meaning and the quantization of these curves in terms of quantum curves has recently received a lot of attention (see for instance \cite{BorA,Dij,DoM,Mul,Guche,Gulec,Gu}). Quantum curves appear naturally in the study of many enumerative problems of algebraic geometry. For example the exponential generating  function of Gromow-Witten invariants on $\mathbb{P}^1$ is the solution of a certain quantum curve computed in \cite{DNMPS}. They also play a key role in a conjecture relating classical knot invariants as the $A$-polynomial and quantum knot invariant as the Jones or the HOMFLY polynomial. More precisely, the quantization of the $A$-polynomial via the topological recursion of Eynard and Orentin \cite{EO}--a recursive procedure conjecturally related to WKB expansion--should allow to recover the colored Jones polynomials. It is worth noticing that though, an intrisic definition of quantum curves quantizing spectral curves of Higgs bundles has been proposed in \cite{Dij}, this definition is not able to capture many of the instances of quantum curves. Indeed, in this definition, quantum curves are interpreted in terms of modules over the Rees algebra of the sheaf of holomorphic differential operators filtered by the order. Such objects allow only to quantize subvarieties of a cotangent bundle and it is well-known that not all spectral curves are subvarieties of a cotangent bundle.      
For instance spectral curves defined by $A$-polynomials lie inside the symplectic surface $(\C^\ast \times \C^\ast, dx_1 \wedge dx_2 /(x_1x_2))$  (\cite{BorA,dim,Gu}). Thus, they are Lagrangian subvarieties of this symplectic manifold. More generally, spectral curves can also be considered as Lagrangian subvarieties of holomorphic symplectic surfaces. This aspect raises naturally the questions of the quantization of spectral curves from the point of view of deformation quantization and the relation between quantum curves and deformation quantization. 

In this paper, we study the quantization of spectral curves from the standpoint of deformation quantization and more specifically from the point of view of Deformation Quantization modules (DQ-modules) and suggest to define quantum curves as certain type of DQ-modules. Indeed, the quantization of spectral curves is a special instance of the problem of quantizing a Lagrangian subvariety inside an holomorphic symplectic manifold for which the theory of DQ-modules--introduced by Kontsevich in \cite{Kos} and thoroughly studied by Kashiwara and Schapira in \cite{KS3}--provides an adequate framework (see \cite{GinPech} and \cite{DS}). In this setting, quantum curves are interpreted as DQ-modules. Our approach deals with the various type of quantum curves (for instance those given by differential operators, translation operators  which control the  generating function on Gromov-Witten invariants on $\mathbb{P}^1$ or scaling operators arising from the quantization of $A$-polynomials) in a uniform way which should allow to set up and study duality between various type of quantum curves. This point of view provides other benefits. For instance the localisation with respects to the deformation parameters of the sheaf of solutions of a quantum curve (understood as a DQ-modules) is a perverse sheaf. 

This paper is divided into three parts. In the first one, we briefly review  the theory of DQ-modules and present some examples of star-algebras which are related to quantum curves. Then, we describe the canonical quantization of the cotangent bundle of a complex manifold as constructed by Polesello and Schapira in \cite{PolSch}. In the second one, we study the quantization of spectral curves  associated to Higgs bundles via DQ-modules theory. We establish Theorem \ref{thm:main}, the main result of this paper, which states, in particular, that given a compact Riemann surface of genus at least two and a Higgs bundle, there always exists an holonomic DQ-module supported by the spectral curve associated to this bundle and that if the spectral curve is smooth and the Higgs bundle is of rank greater that one, this holonomic module is simple which implies that this quantization is locally unique. Finally, in the third one, we set-up a general framework to relate quantum curves and DQ-modules and show that many examples of quantum curves can be interpreted as DQ-modules and that they have the same sheaves of solutions. It is worth noticing that a systematic comparison between DQ-modules and quantum curves is difficult since there is no general theory of quantum curves per se. 

Let us describe the second and third part of this paper in more detail. The second is motivated by the paper \cite{Mul}in which the formulation of the topological recursion for spectral curves in the cotangent bundle of an arbitrary Riemann surface suggests that it should be possible to produce a canonical quantization of a spectral curve associated to a Higgs bundle via the topological recursion (One of the issues is the globalization of the quantization provided locally by the topological recursion). Here, we focus on the existence of a global quantization of the spectral curve without trying to elucidate the relation with topological recursion. We prove that under some mild assumptions, a quantization always exists and that it is locally unique (see Theorem \ref{thm:main}). Our result proves the existence of the quantization of the spectral curve associated to a Higgs bundle in great generality and clarifies certain aspects of \cite{Mul} which uses the language of Rees D-modules to study the quantization of spectral curves associated to Higgs bundles. If one uses Rees D-modules to quantize a spectral curve, a technical difficulty arises from the fact that this curve will not be the support of a Rees D-modules quantizing it but will be its semi-classical characteristic variety (a suitable non-conic version of  the characteristic variety). In the framework of the theory of DQ-modules, in order to quantize spectral curves one has to first choose a quantization of the symplectic surface i.e a DQ-algebra whose associated symplectic structure is the symplectic structure of the surface considered. In the case of the cotangent bundle, we use its canonical quantization constructed by Polesello and Schapira. Then quantizing a spectral curve amounts to construct a coherent DQ-module without $\hbar$-torsion the support of which is the spectral curve. The formulation in the language of DQ-modules of the problem of the existence of a quantization of a spectral curve  allows to rephrase it in terms of cohomology of sheaves (see Proposition \ref{prop:quantization}) and makes the proofs simpler.

In the third part of the paper, we compare DQ-modules and quantum curves. For that purpose, following \cite{BorA,Dij,DNMPS,Gu,MulHur}, we interpret quantum curves as sections of certain algebras of operators (these algebras of operators should be understood as the analogue of the Rees algebra of differential operators filtered by the order in setting where the symplectic surface being quantized is not a cotangent bundle). To achieve the comparaison with DQ-modules, we introduce a notion of polarization of a DQ-algebra which is the counterpart in the language of DQ-modules of the notion of polarization used in the framework of quantum curves. Using this notion of polarization, we show that the different algebras of operators corresponding to various type of quantum curves embed in a flat way in appropriate DQ-algebras that we have previously defined (see Proposition \ref{prop:DQflatDH}, \ref{prop:flatS}, \ref{prop:flatT}). Finally, we establish that the sheaf of solutions associated to a quantum curve is isomorphic to the sheaf of solutions of the DQ-module canonically associated to it (see Corollary \ref{cor:SolRD}, \ref{cor:SolS} and \ref{cor:SolT}). 

\noindent \textbf{Acknowledgement:}
We would like to thanks Motohico Mulase for his hospitality and many useful discussions as well as his encouragement to study quantum curves from the standpoint of DQ-modules theory. We are also grateful to Pierre Schapira for many insightful discussions and scientific advice. We also wish to thank Christopher Dodd for many interesting discussions, Ga\"etan Borot for organizing the workshop \textit{Geometric quantization and topological recursion} and interesting conversations, St\'ephane Korvers, Martin Schlichenmaier, Jean-Marc Schlenker, Nicolas Tholozan, J\'er\'emy Toulisse and Yannick Voglaire for their involvement in the \textit{Working group on Geometry of moduli spaces and topological recursion}.

\section{DQ-algebras and DQ-modules}
In this section, we review some classical facts concerning DQ-algebras and DQ-modules. For a detailed study of these objects we refer the reader to \cite{KS3}.

\subsection{DQ-algebras}

We denote by $\C^\hbar$ the ring of formal power series with complex coefficients in the variable $\hbar$ and by $\C^{\hbar,loc}$ the field of formal Laurent series. Let $(X,\cO_X)$ be a complex manifold. We define the following sheaf of $\C^\hbar$-algebras
\begin{equation*}
\cO_X^\hbar:=\varprojlim_{n \in \N} \cO_X \te_\C (\C^\hbar / \hbar^n \C^\hbar).
\end{equation*}

\begin{definition}
A star-product denoted $\star$ on $\cO_X^\hbar$ is a $\C^\hbar$-bilinear associative multiplication law satisfying
\begin{equation*}
f \star g = \sum_{i \geq 0} P_i(f,g) \hbar^i \;\; \textnormal{for every} \;f, \; g \in \cO_X,
\end{equation*}
where the $P_i$ are holomorphic bi-differential operators such that for every $f, g \in \cO_X, P_0(f,g)=fg$ and 
$P_i(1,f)=P_i(f,1)=0$ for $i>0$. The pair $(\cO_X^\hbar, \star)$ is called a star-algebra. 
\end{definition}

\begin{example}\label{ex:Moyal}
Consider the symplectic holomorphic manifold $X=T^\ast\C^{n}$ endowed with the symplectic coordinate system $(x;u)$ with $x=(x_1,\ldots,x_n)$ and  $u=(u_1,\ldots,u_n)$. It can be quantized by the following star-algebra $(\cO_{X} ^\hbar, \star)$ where
\begin{equation*}
f \star g = \sum_{\alpha \in \N^n} \dfrac{\hbar^{|\alpha|}}{\alpha !} (\partial^\alpha_u f) (\partial^\alpha_x g).
\end{equation*}
In particular, on $T^\ast \C$ we get
\begin{equation*}
f \star g = \sum_{k \geq 0} \dfrac{\hbar^{k}}{k!} (\partial^k_u f) (\partial^k_x g).
\end{equation*}
\end{example}

\begin{example}\label{ex:expexp}
The symplectic manifold $X=(\C^\ast \times \C^\ast, \frac{dx_1 \wedge dx_2}{x_1 x_2})$ can be quantized by the following star-algebra $(\cO_{X} ^\hbar, \star)$ with
\begin{align*}
f \star g &= \sum_{k \geq 0} \dfrac{\hbar^{k}}{k!} (x_2 \partial_{x_2})^k(f) (x_1 \partial_{x_1})^k(g)\\
&=\sum_{k \geq 0} \left( \dfrac{\hbar^{k}}{k!} \left( \sum_{l=0}^k x_2^l \mathcal{S}_k^{(l)} \partial_{x_2}^{l} f \right) \left( \sum_{p=0}^k x_1^p \mathcal{S}_k^{(p)} \partial_{x_1}^{p} g \right) \right).
\end{align*}
where $\mathcal{S}_k^{(l)}$ is the number of ways of partitioning a set of $k$ elements into $l$ nonempty sets. These numbers are called the Stirling numbers of the second kind.
Notice that $x_1 \star x_2= e^\hbar (x_2 \star x_1)$. 
\end{example}

\begin{example}\label{ex:affexp}
The symplectic manifold $X=(\C^\ast \times \C, \frac{dx_1 \wedge dx_2}{x_1})$ can be quantized by the following star-algebra $(\cO_{X} ^\hbar, \star)$ with
\begin{align*}
f \star g &= \sum_{k \geq 0} \dfrac{\hbar^{k}}{k!}  \partial_{x_2}^k(f) (x_1 \partial_{x_1})^k(g)\\
&= \sum_{k \geq 0} \left( \dfrac{\hbar^{k}}{k!}  \left(\partial_{x_2}^{k} f  \right) \left( \sum_{l=0}^k x_1^l \mathcal{S}_k^{(l)} \partial_{x_1}^{l} g \right) \right).
\end{align*}
\end{example}

\begin{definition}
A DQ-algebra $\cA_X$ on $X$ is a $\C_X^\hbar$-algebra locally isomorphic to a star-algebra as a $\C_X^\hbar$-algebra.
\end{definition}

There is a unique $\C_X$-algebra isomorphism $\cA_X / \hbar \cA_X \stackrel{\sim}{\longrightarrow} \cO_X$. We write $\sigma_0: \cA_X \twoheadrightarrow \cO_X$ for the epimorphism of $\C_X$-algebras defined by
\begin{equation*}
\cA_X \to \cA_X / \hbar \cA_X \stackrel{\sim}{\longrightarrow} \cO_X.
\end{equation*}
This induces a Poisson bracket $\lbrace \cdot,\cdot \rbrace$ on $\cO_X$ defined as follows:
\begin{equation*}
\textnormal{for every $a, \; b \in \cA_X$}, \; \lbrace \sigma_0(a),\sigma_0(b) \rbrace=\sigma_0(\hbar^{-1}(ab-ba)).
 \end{equation*}
A DQ-algebra $\cA_X$ can be endowed with the canonical filtration defined by
\begin{equation}\label{filt:DQalg}
\cA_X(k)= 
\begin{cases}
\hbar^{-k} \cA_X & \textnormal{if $k<0$},\\
 \cA_X & \textnormal{if $k \geq 0$}.
\end{cases}
\end{equation}
Note that for every $k \leq 0$, $\cA_X(k) / \cA_X(k-1) \simeq \cO_X$.
\begin{notation}
\noindent (i) If $\cA_X$ is a DQ-algebra, we denote by $\cA_{X^a}$ the opposite algebra $\cA_X^{\opp}$. This algebra is still a DQ-algebra.

\noindent (ii) If $\cA_X$ is a DQ-algebra, we set $\cA_X^{loc}:=\C^{\hbar,loc} \te_{\C^\hbar} \cA_X$.
\end{notation}

We recall Proposition 2.2.12 from \cite{KS3} which allow to construct star-algebras. We follow closely their presentation. We denote by $\cD_X$ the sheaf of holomorphic differential operators on $X$ and set 
\begin{equation*}
\cD_X^\hbar \colon= \varprojlim_{n \in \N} \cD_X \te_\C (\C^\hbar / \hbar^n \C^\hbar).
\end{equation*}

Given a star algebra $\cA_X:=(\cO_X^\hbar,\star)$, there are two $\C^\hbar$-linear morphisms
\begin{align*}
\Phi^l \colon \cO_X^\hbar \to \cD^\hbar && \Phi^r \colon \cO_X^\hbar \to \cD^\hbar\\
 f \mapsto f \star (\cdot)                &&  f \mapsto (\cdot) \star f.
\end{align*}

Let $(x_1,\ldots,x_n)$ be a local coordinate system on $X$ and for $1 \leq i \leq n$ we set

\begin{equation*}
A_i \colon = \Phi^l(x_i) \qquad \textnormal{and} \qquad  B_i \colon = \Phi^r(x_i).
\end{equation*}

The $A_i$ and $Bi$ with $1 \leq i \leq n$ are sections of $\cD_X^\hbar$ satisfying the following conditions

\begin{equation}\label{eq:quantcond}
\begin{cases}
A_i(1)=B_i(1)=x_i,\\
A_i \equiv x_i \hspace{-0.2cm} \mod \hbar \cD_X^\hbar, \; B_i \equiv x_i \hspace{-0.2cm} \mod \hbar \cD_X^\hbar,\\
[A_i, B_j]=0 \;(i,j=1,\ldots,n).
\end{cases}
\end{equation}
Reciprocally,
\begin{proposition}\label{prop:makestaralg}
Let $\lbrace A_i, B_j \rbrace_{1 \leq i,j \leq n}$ be sections of $\cD_X^\hbar$ satisfying conditions \eqref{eq:quantcond}. It defines $\cA_X \subset \cD_X^\hbar$ by
\begin{equation*}
\cA_X= \lbrace a \in \cD_X^\hbar; [a,B_i]=0,i=1,\ldots,n \rbrace
\end{equation*}
and define the $\C_X^\hbar$-linear map $\psi:\cA_X \to \cO_X^\hbar$, $a \mapsto a(1)$. Then,
\begin{enumerate}[(a)]
\item $\psi$ is a $\C_X^\hbar$-linear isomorphism,
\item the product on $\cO_X^\hbar$ given by $\psi(a) \star \psi(b):= \psi(a \cdot b)$ is a star-product and $\psi^{-1}:\cO_X^\hbar \to \cA_X$ is such that
\begin{equation*}
\psi^{-1}(f)\psi^{-1}(g)=\sum_{i \geq 0} \psi^{-1}(P_i(f,g)) \hbar^i \quad \textnormal{for every $f,g \in \cO_X$}
\end{equation*}
where the $P_i$ are bidifferential operators such that
\begin{equation*}
f \star g=\sum_{i \geq 0} P_i(f,g) \hbar^i \quad \textnormal{for every $f,g \in \cO_X$}.
\end{equation*}
\item  The algebra $\cA_X^{\opp}$ is obtained by replacing $A_i$ with $B_i$ for $ 1 \leq i\leq n$ in the above construction. 

\item \label{it:rep} Let $(y_1,\ldots,y_n)$ be a local coordinate system on a copy of $X$. The section $y_i-A_i \in \cD_{X \times Y}^\hbar$ are invertible on $\lbrace x_i \neq y_i \rbrace$ and if $f \in \cO_X$, the section
\begin{equation*}
G(f)=\frac{1}{(2i \pi)^n}\oint f(y) (y_1-A_1)^{-1} \ldots (y_n-A_n)^{-1} dy_1\ldots dy_n
\end{equation*} 
is such that for all $1 \leq i \leq n$, $[G(f),B_i]=0$ and $\psi(f) \equiv f \mod \hbar$.
\end{enumerate}
\end{proposition}
Notice that point \eqref{it:rep} is extracted from the proof of \cite[Proposition 2.2.12]{KS3}.

\subsubsection{The canonical quantization of the cotangent bundle}

Let $M$ be a complex manifold. The cotangent bundle of $M$, that is $X:=T^\ast M$ is endowed with the filtered, conic sheaf of $\C_X$-algbera $\widehat{\cE}_M$ of formal microdifferential operators and its subsheaf $\widehat{\cE}_X(0)$ of operators of order $m \leq 0$. These sheaves have been introduced in \cite{SKK} and we refer the reader to \cite{Sch} for an introduction.

On $X$ there is DQ-algebra $\chW_X(0)$. It has been constructed in \cite{PolSch} and we follow their presentation. 

We consider the complex line $\C$ endowed with the coordinate $t$ and denote by $(t;\tau)$ the associated symplectic coordinate on $T^\ast \C$. We set
\begin{equation*}
\chE_{T^\ast ( M \times \C), \hat{t}}(0)=\lbrace P \in \chE_{T^\ast  M}; [P, \partial_t]=0 \rbrace.
\end{equation*}
We consider the open subset of $T^\ast (M \times \C)$ given in local coordinate by
\begin{equation*}
T_{\tau \neq 0}^\ast(M\times \C)=\lbrace (x,t;\xi,\tau) \in T^\ast (M \times \C) | \tau \neq 0 \rbrace
\end{equation*}
and the map given in local coordinate by
\begin{equation*}
\rho \colon T_{\tau \neq 0}^\ast(M\times \C) \to T^\ast M, \; (x,t;\xi,\tau) \mapsto (x;\xi/ \tau). 
\end{equation*}
We obtain a $\C_X^\hbar$-algebra by considering the ring
\begin{equation*}
\chW_X(0) \colon=\rho_\ast (\chE_{T^\ast ( M \times \C)})(0)|_{ T_{\tau \neq 0}^\ast(M\times \C)})
\end{equation*}
where $\hbar$ acts as $\tau^{-1}$.
If $P$ is a section of $\chW_X(0)$, it can be written in a local symplectic coordinate system $(x_1,\ldots,x_n,u_1,\ldots,u_n)$ as
\begin{equation*}
P=\sum_{j \leq 0} f_j(x,u_i) \tau^j, \, f_j \in \cO_X,\; j \in \Z.
\end{equation*}
Setting $\hbar=\tau^{-1}$, we get
\begin{equation*}
P=\sum_{k \geq 0} f_k(x,u_i) \hbar^k, \, f_k \in \cO_X,\; k \in \N.
\end{equation*}

We denote by $\chW_X$ the localization of $\chW_X(0)$ with respect to $\hbar$. There is the following commutative diagram of natural morphisms of algebras.
\begin{equation*}
\xymatrix{
\pi^{-1} \cD_M \ar@{^{(}->}[r] & \chE_X \ar@{^{(}->}[r]^-{\phi} & \chW_X\\
 \pi^{-1}\cO_M \ar@{^{(}->}[r] \ar@{^{(}->}[u] \ar@{^{(}->}[r]& \chE_X(0) \ar@{^{(}->}[u] \ar@{^{(}->}[r] & \chW_X(0) \ar@{^{(}->}[u]
}
\end{equation*}
where the algebra map $\phi:\chE_X \to  \chW_X$ is given in a local symplectic coordinate $(x_1,\ldots,x_n,u_1,\ldots,u_n)$ system by $x_i \mapsto x_i$, $\partial_{x_i} \mapsto  \hbar^{-1} u_i$.

\subsection{DQ-modules}
Let $(X,\cO_X)$ be a complex manifold endowed with a DQ-algebra $\cA_X$. We denote by $\Mod(\cA_X)$ the Grothendieck category of $\cA_X$-modules and by $\Der(\cA_X)$ its derived category. We write $\Mod_\coh(\cA_X)$ for the Abelian full subcategory of $\Mod(\cA_X)$ the objects of which are the coherent modules over $\cA_X$.

Recall that if $\cA_X$ is a DQ-algebra, then $\cA_X  / \hbar \cA_X \simeq \cO_X$. This provides a functor
\begin{equation*}
\gr_\hbar:\Der(\cA_X) \to \Der(\cO_X), \; \cM \mapsto \cO_X \Lte_{\cA_X} \cM.
\end{equation*}
There is also a functor
\begin{equation*}
(\cdot)^{loc} \colon \Der(\cA_X) \to \Der(\cA^{loc}_X), \; \cM \mapsto \cM^{loc}:=\C^{\hbar,loc} \te_{\C^\hbar} \cM.
\end{equation*}

We will need the following result which is a special case of \cite[Theorem 1.2.5]{KS3}. 

\begin{theorem}\label{thm:annulationcohoco}
For any coherent $\cA_X$-module $\cM$ and any Stein open subset $U$ of $X$, we have $\Hn^j(U,\cM)=0$ for any $j>0$.
\end{theorem} 

\begin{proposition}{\cite[Corollary 2.3.4 and Corollary 2.3.18]{KS3}}
\begin{enumerate}[(i)]
\item Let $\cM \in \Der^{\mathrm{b}}_\coh(\cA_X)$. Then $\Supp(\cM)$ is a closed analytic subset of $X$.

\item Let $\cM \in \Der^{\mathrm{b}}_\coh(\cA^{loc}_X)$. Then $\Supp(\cM)$ is a closed analytic subset of $X$, coisotropic for the Poisson bracket on $X$ associated with $\cA_X$.
\end{enumerate}
\end{proposition}

We recall the definition of simple module (see \cite[Definition 2.3.10]{KS3}). 

\begin{definition}
Let $\Lambda$ be a smooth submanifold of $X$ and let $\cL$ be a coherent $\cA_X$-module supported by $\Lambda$. The module $\cL$ is simple along $\Lambda$ if $\gr_\hbar(\cL)$ is concentrated in degree zero and $\Hn^0(\gr_\hbar \cL)$ is an invertible $\cO_\Lambda$-module.
\end{definition}

When the associated Poisson structure of $\cA_X$ is symplectic, we have the following notions and results.
\begin{lemma}[{\cite[Lemma 6.2.1]{KS3}}] \label{lem:locuni}
Assume that the Poisson structure associated with $\cA_X$ is symplectic. Let $\Lambda$ be a smooth Lagrangian submanifold of $X$ and let $\cL_i$ $(i=0, \; 1)$ be simple $\cA_X$-modules along $\Lambda$. Then:
\begin{enumerate}[(i)]
\item the simple modules $\cL_0$ and $\cL_1$ are locally isomorphic,

\item the natural morphism $\C^\hbar_X \to \fHom_{\cA_X}(\cL_0, \cL_0)$ is an isomorphism.
\end{enumerate}
\end{lemma}
\begin{definition}
\begin{enumerate}[(i)]
\item an $\cA_X$-module $\cM$ has no $\hbar$-torsion if the map $\cM \stackrel{\hbar}{\to} \cM$ is a monomorphism.

\item An $\cA_X^{loc}$-module $\cM$ is holonomic if it is coherent and if its support is a Lagrangian subvariety of $X$.

\item An $\cA_X$-module $\cN$ is holonomic if it is coherent, without $\hbar$-torsion and $\cN^{loc}$ is a holonomic $\cA_X^{loc}$-module.
\end{enumerate}
\end{definition}

We denote by $\Der^b_{\C \mathrm{c}}(\C^{\hbar,loc}_X)$ the full subcategory of $\Der^{b}(\C^{\hbar,loc}_X)$ consisting of $\C$-constructible objects (see \cite[Definition 8.5.6 (ii)]{KS1}). There is the following important result concerning holonomic $\cA_X^{loc}$-modules.

\begin{theorem}[{\cite[Theorem 7.2.3]{KS3}}] Let $X$ be a complex symplectic manifold of dimension $d_X$ and let $\cM$ and $\cL$ be two holonomic $\cA_X^{loc}$-modules. Then,
\begin{enumerate}[(i)]
\item the object $\fRHom_{\cA_X^{loc}}(\cM,\cL)$ belongs to $\Der^b_{\C \mathrm{c}}(\C^{\hbar,loc}_X)$,
\item there is a canonical isomorphism:
\begin{equation*}
\fRHom_{\cA_X^{loc}}(\cM,\cL) \stackrel{\sim}{\to} \fRHom_{\C_X^{\hbar,loc}}(\fRHom_{\cA_X^{loc}}(\cL,\cM),\C_X^{\hbar,loc})[d_X],
\end{equation*} 
\item the object $\fRHom_{\cA_X^{loc}}(\cM,\cL)[d_X/2]$ is perverse.
\end{enumerate} 
\end{theorem}

\section{Quantization of spectral curves}

In this section, we prove that it is possible to quantize, in the framework of DQ-modules, the spectral curve associated to a Higgs bundle on a Riemann surface of genus at least 2.

If $E$ is a vector bundle on a complex manifold $X$ and $s\colon X \to E$ is a section of $E$, we denote be $X_s$ the analytic zero subscheme of $s$ and by $\Zl(s)$ the set theoretic zero locus of $s$.

By a Riemann surface we mean a connected, one dimensional complex manifold. We do not assume that Riemann surfaces are compact.

In all this section $(\Sigma, \cO_\Sigma)$ is a Riemann surface and we set $X=T^\ast \Sigma$ and $\pi:X \to \Sigma$ the canonical projection on the base. Let $\cL$ be a line bundle over $\Sigma$. We set
\begin{equation*}
\chW_X^{\cL}(0):= \chW_X(0) \te_{\pi^{-1}\cO_\Sigma} \pi^{-1} \cL.
\end{equation*}
Notice that the sheaf $\chW_X^{\cL}(0)$ is a $\chW_X(0)$-module without $\hbar$-torsion. It is $\hbar$-complete and coherent. Recall that we have the following exact sequence
\begin{equation}\label{seq:init}
0 \to  \hbar \chW_X(0) \to \chW_X(0) \to \cO_X \to 0
\end{equation}
where $\hbar \chW_X(0)$ denotes the image of the morphism $\chW_X(0) \stackrel{\hbar}{\to} \chW_X(0)$.

Applying the exact functor $(\cdot) \te_{\pi^{-1}\cO_\Sigma} \pi^{-1} \cL$ to the sequence \eqref{seq:init}, we get
\begin{equation}\label{seq:Ltwist}
0 \to  \hbar \chW^{\cL}_X(0) \to \chW^{\cL}_X(0) \stackrel{\sigma_0}{\to} \pi^\ast {\cL} \to 0.
\end{equation}
It follows from the above sequence that $\chW^{\cL}_X(0) / \hbar \chW^{\cL}_X(0) \simeq \pi^\ast {\cL}$.

We state and prove the following quantization criterion, that we will apply repeatedly.
\begin{proposition}\label{prop:quantization}
Let $\Sigma$ be a Riemann surface and let $\cL$ be a line bundle on $\Sigma$ and $s \neq 0$ be a section of $\pi^\ast \cL$. Assume that $\Hn^1(X,\chW_X^{\cL}(0))=0$. Then, there exists a coherent $\chW_X(0)$-module $\cM$ without $\hbar$-torsion supported by $\Zl(s)$ and such that $\cM / \hbar \cM \simeq \pi^\ast \cL /(s)$ where $(s)$ denotes the $\cO_X$-submodule of $\pi^\ast \cL$ generated by $s$. 
\end{proposition}
\begin{proof}
Using the long exact sequence for the functor $\Gamma(X; \cdot)$ applied to the exact sequence \eqref{seq:Ltwist}, we obtain the exact sequence
\begin{equation*}
 \Hn^0(X,\chW^{\cL}_X(0)) \stackrel{\sigma_{0}}{\to} \Hn^0(X,\pi^\ast {\cL}) \to \Hn^1(X,\hbar \chW_X^{\cL}(0)).
\end{equation*}
The morphism $\chW_X^{\cL}(0) \stackrel{\hbar}{\to} \hbar \chW_X^{\cL}(0)$ is an isomorphism of $\C^\hbar$-modules and the group $\Hn^1(X,\chW_X^{\cL}(0))=0$. Then, $\Hn^1(X,\hbar \chW_X^{\cL}(0))=0$. This implies that the morphism $\sigma_0$ is an epimorphism. It follows that we can find $s^\hbar \in \chW^{\cL}_X(0)(X)$ such that $\sigma_0(s^\hbar)=s$. We denote by $(s^\hbar)$ the left $\chW_X(0)$-submodule of $\chW^{\cL}_X(0)$ generated by $s^\hbar$ and set
\begin{equation*}
\cM:=\chW_X^{\cL}(0) / (s^\hbar).
\end{equation*}
The module $\cM$ has no $\hbar$-torsion since $\sigma_0 (s^\hbar) \neq 0$ and is coherent since it is a quotient of coherent modules. Finally, since $\cM$ has no $\hbar$-torsion
\begin{equation*}
\Supp(\cM)= \Supp(\cM / \hbar \cM)=\Supp(\pi^\ast \cL / (s))=\Zl(s).
\end{equation*}
\end{proof} 

\begin{remark} There is a straightforward way to quantize the zero locus of a holomorphic function in complex Poisson manifold quantized by a star-algebra. Indeed, if $X$ is a complex Poisson manifold endowed with a star-algebra $(\cO_X^\hbar,\star)$ (as for instance in examples \ref{ex:Moyal}, \ref{ex:expexp} and \ref{ex:affexp}). Then, there exist a $\C$-linear section $\phi: \cO_X \to \cO_X^\hbar$ of $\sigma_0 \colon \cO_X^\hbar \to \cO_X$. Thus, if $f \in \cO_X$,  the $(\cO_X^\hbar,\star)$-module $\cO_X^\hbar / \cO_X^\hbar \phi(f)$ is a coherent $(\cO_X^\hbar,\star)$-module without $\hbar$-torsion supported by $\lbrace x \in X | f(x)=0 \rbrace$.
\end{remark}

\subsection{The case of open Riemann surfaces}

\begin{theorem}\cite{BeS1,BeS2}
An open Riemann surface is a Stein manifold.
\end{theorem}

The following proposition is well-known to experts (see for instance \cite{Forstneric})

\begin{proposition}
If $E \to M$ is a holomorphic vector bundle over a Stein base $M$ then the total space $E$ is also Stein.
\end{proposition}

This implies that the total space of the cotangent bundle to an open Riemann surface is Stein.

\begin{lemma}
If $\Sigma$ is an open Riemann surface then, 
\begin{equation*}
\Hn^1(X,\chW_X^{\cL}(0))=0.
 \end{equation*}
\end{lemma}

\begin{proof}
Since $X$ is Stein and $\chW_X^{\cL}(0)$ is a coherent $\chW_X(0)$-module, the result follows from Theorem \ref{thm:annulationcohoco}.
\end{proof}

\subsection{The case of compact Riemann surfaces}\label{subsec:compact}

In this subsection we give a criterion for the vanishing of the cohomology of $\chW_X^{\cL}(0)$ when $\Sigma$ is compact and use it to quantize spectral curves associated to Higgs bundle.

We will need the following classical lemma due to Grothendieck. 

\begin{lemma}\label{lem:comtensor}
Let $F^\bullet$ be a bounded complex of nuclear Fr\'echet spaces such that for every $i \in \Z$, $\Hn^i(F^\bullet)$ is a finite dimensional vector space. Let $V$ be a Fr\'echet nuclear complex vector space. Then, for every $i \in \Z$,
\begin{equation*}
\Hn^i(F^\bullet \hat{\otimes} V) \simeq \Hn^i(F^\bullet)\otimes V.
\end{equation*}
\end{lemma}

\begin{lemma}\label{lem:projbundle}
Let $M$ be a complex manifold and $\pi: E\to M$ be an holomorphic vector bundle and $\cL$ be a locally free $\cO_M$-module of finite rank. Then,
\begin{equation*}
\Rg(E; \pi^\ast \cL) \simeq \Rg(M;\pi_\ast \cO_E \te_{\cO_M} \cL).
\end{equation*}
\end{lemma}

\begin{proof} We write $a_M: M \to *$ for the map from $M$ to the point.
Then, we have
\begin{align*}
\Rg(E; \pi^\ast \cL)&\simeq \mathrm{R}a_{M \ast} \mathrm{R}\pi_\ast (\cO_E \te_{\pi^{-1}\cO_M} \pi^{-1} \cL)\\
                    &\simeq \Rg(M,\mathrm{R}\pi_\ast (\cO_E) \te_{\cO_M} \cL).
\end{align*}
But $\mathrm{R}\pi_\ast (\cO_E) \simeq \pi_\ast \cO_E$. Indeed, $\mathrm{R}^i\pi_\ast (\cO_E)$ is the sheaf associated to the presheaf $U \mapsto \Hn^i(\pi^{-1}(U),\cO_E)$ and there exists a fundamental system of Stein open sets $(V_i)_{i \in I}$ such that for every $i \in I$, the open set $\pi^{-1}(V_i)$ is Stein.
\end{proof}

\begin{proposition}\label{prop:cohomologypull}
Let $M$ be a compact complex manifold and $\pi: E\to M$ be an holomorphic vector bundle and $\cL$ a locally free $\cO_M$-module of finite rank. Then, for every $i \geq 0$
\begin{enumerate}[(i)]
\item $\Hn^i(M, \pi_\ast \cO_E \te_{\cO_M} \cL)\simeq \Hn^i(M,\cL) \otimes \Gamma(E_x,\cO_{E_x})$,
\item $\Hn^i(E, \pi^\ast\cL)\simeq \Hn^i(M,\cL) \otimes \Gamma(E_x,\cO_{E_x})$
\end{enumerate}
where $x \in M$.
\end{proposition}

\begin{proof}
\noindent (i)  Let $\mathcal{U}=(U_i)_{i \in I}$ be a Stein covering of $M$, such that $E|_{U_i}$ and $\cL|_{U_i}$ are trivial. Then,
\begin{align*}
\Gamma(U_i, \pi_\ast \cO_E \otimes_{\cO_M} \cL) & \simeq (V \hat{\otimes} \cO_M (U_i)) \hat{\otimes}_{\cO_M(U_i)} \cL(U_i)\\
                                                &  \simeq \cL(U_i) \hat{\otimes} V  
\end{align*}
where $V=\Gamma(E_x,\cO_{E_x})$ with $x \in M$.
Using the above isomorphism the \v{C}ech complex of $\pi_\ast \cO_E \otimes_{\cO_M} \cL$ relatively to the covering $\mathcal{U}$ can be written
\begin{equation*}
\cdots \to C^{q-1}(\mathcal{U},\cL) \hat{\otimes} V \to C^q(\mathcal{U},\cL) \hat{\otimes} V \to C^{q+1}(\mathcal{U},\cL) \hat{\otimes} V \to \cdots. 
\end{equation*}

It follows from Leray's Theorem that for every $q \geq 0$, $\check{\Hn}^q(\mathcal{U},\cL)\simeq \Hn^q(M,\cL)$. Since the manifold $M$ is compact, this implies that the $\check{\Hn}^q(\mathcal{U},\cL)$ are finite dimensional $\C$-vector spaces. It follows from Lemma \ref{lem:comtensor} that  
\begin{equation*}
\textnormal{for every}\; q \geq 0, \; \check{\Hn}^q(\mathcal{U},\cL \hat{\otimes} V) \simeq \check{\Hn}^q(\mathcal{U},\cL) \otimes V.
\end{equation*}
This, in turns, implies via Leray's Theorem the following isomorphism.
\begin{equation*}
 \textnormal{for every $q \geq 0$},\; \Hn^q(M, \pi_\ast \cO_E \te_{\cO_M} \cL)\simeq \Hn^q(M,\cL) \otimes V.
\end{equation*}

\noindent (ii) It is a direct consequence of Lemma \ref{lem:projbundle} and Proposition \ref{prop:cohomologypull} (i).

\end{proof}

For every $n \in \N$, we set $\cN_n:=\chW_X^{\cL}(0) / \hbar^n \chW_X^{\cL}(0)$.

\begin{lemma}\label{lem:annulationtronc} Let $\Sigma$ be a compact Riemann surface and $\cL$ be a line bundle such that $\Hn^1(\Sigma,\cL)=0$. Then, for every $i>0$ and $n \geq 0$, $\Hn^i(X,\cN_n)\simeq 0$. 
\end{lemma}

\begin{proof} The case $n=0$ is clear.
We prove by induction on $n>0$ that 
\begin{equation*}
\textnormal{for every $i>0$ and $n>0$,}\; \Hn^i(X,\cN_n)\simeq 0. 
\end{equation*} 
Since $\cN_1 \simeq \pi^\ast \cL$, we know that $\Hn^i(X,\cN_1)\simeq 0$ for every $i>0$. Thus, the case $n=1$ is settled. Assume that $n \geq 2$ and that $\Hn^i(X,\cN_{n-1})\simeq 0$ for every $i>0$. For every $n \geq 2$, we have the following exact sequence
\begin{equation*}
 0 \to \cN_{n-1} \stackrel{\hbar}{\to} \cN_{n} \to \cN_1 \to 0.
\end{equation*}
For $i>0$, we get the following exact sequence
\begin{equation*}
\Hn^i(X,\cN_{n-1}) \to \Hn^i(X,\cN_{n}) \to \Hn^i(X,\cN_{1}) \to \Hn^{i+1}(X,\cN_{n-1}).
\end{equation*}
We know that all the terms of the above sequence but maybe $\Hn^i(X,\cN_{n})$ are zero. This implies that $\Hn^i(X,\cN_{n})\simeq 0$ for $i>0$.
\end{proof}

\begin{lemma}\label{lem:annulationwl} Let $\Sigma$ be a compact Riemann surface and $\cL$ be a line bundle such that $\Hn^1(\Sigma,\cL)=0$. Then,
\begin{equation*}
\Hn^i(X,\chW_X^{\cL}(0)) \simeq 0, \; \textnormal{for} \; i>0.
\end{equation*}
\end{lemma}

\begin{proof} Recall that for every $n \in \N$, $\cN_n:=\chW_X^{\cL}(0) / \hbar^n \chW_X^{\cL}(0)$  and that $\chW_X^{\cL}(0) \simeq \varprojlim \limits_n \cN_n$. For every $n \geq 0$, the natural map $\cN_{n+1} \to \cN_n$ is an epimorphism of sheaves.  

We wish to prove that for every $i \geq 0$ the canonical map
\begin{equation}\label{map:MLWL}
h_i \colon \Hn^i(X,\chW_X^{\cL}(0))\to \varprojlim_n \Hn^i(X,\cN_n)
\end{equation}
is an isomorphism.
According to \cite[Proposition 13.3.1]{EGAIII1}, it is sufficient for us to check that for every $i \geq 0$ the projective system $(\Hn^i(X,\cN_n))_{n \in \N}$ satisfies the Mittag-Leffler's condition. For $i>0$, this is clear since $\Hn^i(X,\cN_n)=0$. For $i=0$, it is sufficient to check that the morphism 
\begin{equation}\label{map:MLsur}
\Gamma(X; \chW_X^{\cL}(0) / \hbar^{n+1} \chW_X^{\cL}(0)) \to \Gamma(X;\chW_X^{\cL}(0) / \hbar^{n} \chW_X^{\cL}(0))
\end{equation}
is surjective for any $n \geq 1$.
Consider the following short exact sequence
\begin{equation*}
0 \to \pi^\ast \cL \stackrel{\hbar^n}{\to} \chW_X^{\cL}(0) / \hbar^{n+1} \chW_X^{\cL}(0) \to \chW_X^{\cL}(0) / \hbar^{n} \chW_X^{\cL}(0) \to 0.
\end{equation*}
It provides us with the following exact sequence
\begin{equation*}
\Gamma(X; \chW_X^{\cL}(0) / \hbar^{n+1} \chW_X^{\cL}(0)) \to \Gamma(X;\chW_X^{\cL}(0) / \hbar^{n} \chW_X^{\cL}(0)) \to \Hn^1(X,\pi^\ast \cL).
\end{equation*}
It follows from Proposition \ref{prop:cohomologypull} (ii) that,
\begin{equation*}
\Hn^1(X, \pi^\ast \cL) \simeq 0. 
\end{equation*}
This implies that morphism \eqref{map:MLsur} is surjective. Since the morphism \eqref{map:MLWL} is an isomorphism for every $i \geq 0$, the result follows from Lemma \ref{lem:annulationtronc}.
\end{proof}

\subsection{Application}
We now apply the results of Subsection \ref{subsec:compact} to the following case. We assume that $\Sigma$ is a compact Riemann surface of genus $g \geq 2$.

Once again we set $X:=T^\ast \Sigma$ and write $\pi: X \to \Sigma$ for the canonical projection. We denote by $\eta$ the Liouville form on $X$, i.e the tautological section of $T^\ast \Sigma \to T^\ast X$. If $\cL$ is a line bundle on $\Sigma$ we denote by $\deg \cL$ its degree.

We recall a few well-know facts about compact Riemann surfaces.   

\begin{proposition}\label{prop:annulationdeg}\begin{enumerate}[(i)]
\item $\deg \Omega^1_\Sigma=2g-2$.
\item Let $\cL$ be an holomorphic line bundle such that $\deg \cL > \deg \Omega^1_\Sigma$. Then, $\Hn^1(\Sigma,\cL)=0$.
\end{enumerate}
\end{proposition}

\begin{definition}
A Higgs bundle is a pair $(\cE,\phi)$ where $\cE$ is a locally free $\cO_\Sigma$-module of finite rank and $\phi \in \Hn^0(\Sigma, \fEnd(\cE) \otimes \Omega^1_\Sigma)$. The section $\phi$ is called the Higgs field of $(\cE, \phi)$. 
\end{definition}

\begin{remark}
One usually requires that $\phi$ satisfies $\phi \wedge \phi=0$.
\end{remark}

We recall the construction of the characteristic polynomial of a Higgs field. For that purpose, we state the following well-known fact.

\begin{lemma}\label{lem:rap}
Let $(M, \cO_M)$ be a complex manifold and $\cL$ be a locally free $\cO_M$-module of finite rank $r \in \N^\ast$ and $\cL^\prime$ be a locally free sheaf of finite rank on $M$. Then, 
\begin{enumerate}[(i)]
\item $\bigwedge_{i=1}^r \cL$ is a locally free $\cO_M$-module of rank one.

\item There is a canonical isomorphism
\begin{equation}\label{map:identification}
\Hom_{\cO_M}(\bigwedge_{i=1}^r \cL, \bigotimes_{i=1}^r \cL^\prime \otimes \bigwedge_{i=1}^r \cL) \simeq \Gamma(M, \bigotimes_{i=1}^r \cL^\prime).
\end{equation}
\end{enumerate}
\end{lemma}

Let $\cL$ and $\cL^\prime$ be as in Lemma \ref{lem:rap}. Consider a pair $(\cL,\psi:\cL \to  \cL^\prime \otimes \cL)$  where $\psi$ is a morphism of sheaves. We have the following diagram

\begin{equation*}
\xymatrix{ \bigotimes_{i=1}^r \cL  \ar[r]^-{\psi^{\otimes r}} \ar[d] & \bigotimes_{i=1}^r(\cL^\prime \otimes \cL) \ar[r] & \bigotimes_{i=1}^r \cL^\prime \otimes \bigwedge_{i=1}^r \cL\\
\bigwedge_{i=1}^r  \cL. \ar@{-->}[rru]_-{\exists ! s}
}
\end{equation*}
By the universal property of the exterior power, there exists a unique section $s$ filling the dotted arrow in the above diagram. We identify $s$ with a section of $\Gamma(M, \bigotimes_{i=1}^r \cL^\prime)$ via the isomorphism \eqref{map:identification}. We call this section the determinant of $\psi$ and denote it by $\det(\psi)$.

Let $(\cE,\phi)$ be a Higgs bundle on $\Sigma$. The characteristic polynomial of $\phi$ is obtained by applying the above construction with $M=X$, $\cL=\pi^\ast \cE$, $\cL^\prime=\pi^\ast\Omega_\Sigma^{1}$ and $\psi=\pi^\ast \phi - \eta \otimes \id_{\pi^\ast \cE}$. It is the section $\det(\pi^\ast \phi-\eta) \in \Gamma(X;\Omega_\Sigma^{1 \otimes r})$.

\begin{theorem}\label{thm:main}
Let $\Sigma$ be a compact Riemann surface of genus $g \geq 2$. Let $(\cE,\phi)$ be a Higgs bundle of rank $r$ on $\Sigma$.
 \begin{enumerate}[(i)]

\item If $r=1$ there exists a holonomic $\chW_X(0)$-module $\cM$ such that $\cM / \hbar \cM \simeq \pi^\ast \Omega_\Sigma^{\otimes 2} / (\det( \pi^\ast \phi-\eta)\otimes \det( \pi^\ast \phi-\eta))$ (in particular $\Supp(\cM)=\Zl(\det( \pi^\ast \phi-\eta))$.
 
\item If $r \geq 2$, there exists a holonomic $\chW_X(0)$-module $\cM$ such that $\cM / \hbar \cM \simeq \pi^\ast \Omega_\Sigma^{\otimes r} / (\det( \pi^\ast \phi-\eta))$ (in particular $\Supp(\cM)=\Zl(\det( \pi^\ast \phi-\eta))$. Moreover if the analytic space $X_{\det( \pi^\ast \phi-\eta)}$ is smooth, then $\cM$ is a simple $\chW_X(0)$-module. 
\end{enumerate}
\end{theorem}

\begin{proof}
We have that $\deg \Omega_\Sigma^{\otimes n}= n(2g-2)$. If $n=\max(2,r)$ then, it follows from Proposition \ref{prop:annulationdeg} that $\Hn^1(\Sigma,\Omega_\Sigma^{\otimes n})=0$. By Lemma \ref{lem:annulationwl}, we have that $\Hn^1(X,\chW^{\Omega_\Sigma^{\otimes n}}_X(0))=0$. Then, the existence of the module $\cM$ follows immediately from Proposition \ref{prop:cohomologypull}.

Under the assumption of $(ii)$, if the analytic space $X_{\det( \pi^\ast \phi-\eta)}$ is smooth, it follows from the definition of a simple module that $\cM$ is simple.  
\end{proof}

\begin{remark}
\begin{enumerate}[(i)]
\item The $\chW_X(0)$-modules we have constructed are holonomic, since they are coherent, have no $\hbar$-torsion and are supported by Lagrangian subvarieties (i.e. curves in a symplectic surface).

\item If $U$ is an open subset of $X$ and $(x; \xi)$ is a local coordinate system of $X$ on $U$ then $\det( \pi^\ast \phi-\eta)$ can be identified with a polynomial
\begin{equation*}
P(x,\xi)= \xi^n+a_{n-1}(x) \xi^{n-1}+\ldots+a_0(x)
\end{equation*} 
where $a_i \in \cO_U$. Identifying $P(x,\xi)$ with the total symbol of a differential operator we obtain the following DQ-module on $U$, $\cN=\chW_X(0) / \chW_X(0) P$. Moreover, $\cN / \hbar \cN \simeq \cO_U / \cO_U P$ which implies that 
\[
\Supp(\cN)= \lbrace (x, \xi) | P(x, \xi)=0 \rbrace. 
\]
If the module $\cM$ obtained in Theorem \ref{thm:main} (ii) is simple it follows that it is locally isomorphic to a $\chW_X(0)$-modules of the form $\chW_X(0) / \chW_X(0) P$ (see Lemma \ref{lem:locuni}).

\item The result concerning the simplicity of $\cM$ in Theorem \ref{thm:main} (ii) together with Proposition \ref{prop:DQflatDH} clarify the relation between the quantizations constructed in Theorem \ref{thm:main} and the quantization of the spectral curve constructed in \cite{Mul}.

\end{enumerate}
\end{remark}

\section{Relations between Quantum curves and DQ-modules}

The quantization of spectral curves is usually formulated by using the notion of quantum curve. This quantization is usually achieved by quantizing the coordinates of the ambient space. 

In this section, we define DQ-algebras which corresponds to the so-called quantization rules (or polarizations of coordinates) encountered in papers dealing with quantum curves and introduce a notion of polarization of a DQ-algebra which allows us to make the comparison between quantum curves and DQ-modules. 

It is worth mentionning that a quantum curve is often understood in the mathematical physic literature as a "Schr\"odinger-type" operator. As there is no general theory of quantum curves per se, it is difficult to make a systematic comparison between DQ-modules and quantum curves. Here, we perform the comparison in the following sense. It seems that for each major type of quantum curves arising from topological recursion, there exists an algebra of operators such that quantum curves of a given type can be interpreted as modules over this algebra. We prove that these various algebras and their modules can be studied in an uniform way via the theory of DQ-modules since a quantum curve can be interpreted as a DQ-module over a suitable DQ-algebra.

\subsection{Polarizations}

In this subsection, we introduce the notion of polarization of a DQ-algebra and give examples corresponding to the main situations in which one encounters quantum curves. Polarizations will allow us to compare solutions of a quantum curves with the solutions of the DQ-modules canonically associated to them.

\begin{definition}Let $X$ be a complex symplectic manifold and let $\cA_X$ be a DQ-algebra the associated Poisson structure of which is the symplectic structure of $X$. A polarization $\cP$ of $(X, \cA_X)$ is the data of
\begin{enumerate}[(i)]
\item an holomorphic fiber bundle $\pi:X \to \Lambda$ such that $\pi^{-1}(x)$ is a Lagrangian submanifold of $X$,
\item a Lagrangian immersion $\iota:\Lambda \to X$ such that $\pi \circ \iota =\id_{\Lambda}$,
\item A $\cA_X$-module $\cL$ simple along $\Lambda$.
\end{enumerate}
\end{definition}

\begin{remark}
In practice we often have $\cL \simeq \cO_{\Lambda}^\hbar$ as $\C_X^\hbar$-modules.
\end{remark}

Roughly speaking, specifying a polarization i.e "a quantization rule" in the sense of quantum curves (see for example \cite{Dij,Mul,Gu,Norbury}) corresponds to specify the action of $\cA_X$ on $\cO^\hbar_{\Lambda}$. We now describe precisely this correspondance in the case of the cotangent bundle to a complex manifold, in the case of $(\C^\ast \times \C^\ast, (dx_1 \wedge d x_2) /(x_1 x_2))$ and of $(\C^\ast \times \C, (dx_1 \wedge d x_2) / x_1)$.

\begin{definition} Let $X$ be a symplectic manifold and assume it is quantized by a DQ-algebra $\cA_X$. Let $\cP=(\pi \colon X \to \Lambda,\iota: \Lambda \to X,\cL)$ be a polarization of $(X,\cA_X)$. Let $\cM$ be a coherent $\cA_X$-module. The complex of solution of $\cM$ with respect to the polarization $\cP$ is the sheaf $\fRHom_{\cA_X}(\cM, \cL)$.
\end{definition}

\begin{remark}
If $\cM$ is an holonomic module then, $\fRHom_{\cA^{loc}_X}(\cM^{loc}, \cL^{loc})[d_X/2]$ is a perverse sheaf. Thus, if $\cM$ is a DQ-modules quantizing a spectral curve then $\fRHom_{\cA^{loc}_X}(\cM^{Loc}, \cL^{loc})[1]$ is a perverse sheaf. Thus, if we know how to quantize a spectral curve, we can associate to it a perverse sheaf.
\end{remark}

\begin{definition}
Let $X$ be a symplectic manifold and assume it is quantized by a DQ-algebra $\cA_X$. Let $\cP=(\pi \colon X \to \Lambda,\iota: \Lambda \to X,\cL)$ be a polarization of $(X,\cA_X)$. Let $\cR_\Lambda$ be a coherent sheaf of algebras together with a monomorphism of algebra $\phi:\pi^{-1} \cR_\Lambda \hookrightarrow \cA_X$. Let $\cL_\Lambda$ be a $\cR_\Lambda$-module. We say that the polarization $\cP$ and the pair $(\phi, \cL_\Lambda)$ are compatible if
\begin{enumerate}[(i)]
\item  the functor 
\[
(\cdot)^{\cA_X}:\Mod_\coh(\cR_\Lambda) \to \Mod_\coh(\cA_X), \; \cM \mapsto \cA_X \te_{\pi^{-1} \cR_\Lambda} \pi^{-1} \cM
\] is exact, 
\item $\pi_\ast\cL$ and $\cL_\Lambda$ are isomorphic as $\cR_\Lambda$-modules.
\end{enumerate}
\end{definition}

\begin{proposition}\label{prop:compsol}
Let $X$ be a symplectic manifold and assume it is quantized by a DQ-algebra $\cA_X$. Let $\cP=(\pi \colon X \to \Lambda,\iota: \Lambda \to X,\cL)$ be a polarization of $(X,\cA_X)$. Let $\cR_\Lambda$ be a coherent sheaf of algebras together with a monomorphism of algebra $\phi:\pi^{-1} \cR_\Lambda \hookrightarrow \cA_X$. Let $\cL_\Lambda$ be a $\cR_\Lambda$-module and $\cM$ a coherent $\cR_\Lambda$-module. If the pair $(\phi, \cL_\Lambda)$ is compatible with the polarization $\cP=(\pi \colon X \to \Lambda,\iota: \Lambda \to X,\cL)$, then
\begin{equation*}
\fRHom_{\cR_\Lambda}(\cM, \cL_\Lambda) \simeq \mathrm{R} \pi_\ast  \fRHom_{\cA_X}(\cM^{\cA_X}, \cL).
\end{equation*}
\end{proposition}

\begin{proof}
We have the following isomorphisms.
\begin{align*}
\mathrm{R} \pi_\ast  \fRHom_{\cA_X}(\cM^{\cA_X}, \cL) &\simeq \mathrm{R} \pi_\ast  \fRHom_{\pi^{-1} \cR_\Lambda}(\pi^{-1} \cM, \cL) \\
& \simeq \fRHom_{\cR_\Lambda}(\cM, \mathrm{R} \pi_\ast \cL)\\
& \simeq \fRHom_{\cR_\Lambda}(\cM, \cL_\Lambda).
\end{align*}
\end{proof}

We prove a flatness criterion that we will use to study sheaves of solutions.

\begin{theorem}\label{thm:flatcrit}
Let $X$ be a complex manifold endowed with a DQ-algebra $\cA_X$ and let $\cB_X$ be a $\cB_X$-coherent $\C[\hbar]_X$sub-algebra of $\cA_X$. Assume that $\cA_X / \hbar \cA_X$ is flat over $\cB_X / \hbar \cB_X$. Then, $\cA_X$ is flat over $\cB_X$.
\end{theorem}

\begin{proof}
We adapt the proof of \cite[Proposition 5.2.3]{AK}. Since $\cB_X$ is coherent, we just need to prove that for every open set $U \subset X$ and $\cN \in \Mod_\coh(\cB_X|_U)$
\begin{equation*}
\Hn^{-1}(\cA_X|_U \Lte_{\cB_X|_U} \cN)=0.
\end{equation*} 
For the sake of brevity we will omit the restriction to $U$. Since any coherent module is an extension of a module without $\hbar$-torsion by a module of $\hbar$-torsion, it is sufficient to establish the results for these modules. Recall that if $\cM$ is a $\cA_X$-module, we set $\gr_\hbar \cM = \C_X \Lte_{\C[\hbar]_X} \cM$. 
\begin{enumerate}[(a)]
\item Assume $\cN$ has no $\hbar$-torsion. Then,
\begin{align*}
\Hn^{-1}(\gr_\hbar (\cA_X \Lte_{\cB_X}\cN))& \simeq \Hn^{-1}(\cA_X/ \hbar \cA_X \Lte_{\cA_X} \cA_X \Lte_{\cB_X}\cN))\\
                &\simeq \Hn^{-1}(\cA_X/ \hbar \cA_X \Lte_{\cB_X / \hbar \cB_X} \cN/ \hbar \cN).
\end{align*}
By hypothesis, it follows that
\begin{equation*}
\Hn^{-1}(\gr_\hbar (\cA_X \Lte_{\cB_X}\cN)) \simeq 0.
\end{equation*}
Moreover, it follows from Lemma 1.4.2 of \cite{KS3} that we have the exact sequence
\begin{align*}
0\to \cO_X \te_{\cA_X} \fTor_1^{\cB_X}(\cA_X,\cN) \to \Hn^{-1}(\gr_\hbar (\cA_X \Lte_{\cB_X}\cN)) \to \\ 
 \to \fTor_1^{\cA_X}(\cO_X,\Hn^{0}(\cA_X \Lte_{\cB_X} \cN)) \to 0.
\end{align*}
This implies that $\cO_X \te_{\cA_X} \fTor_1^{\cB_X}(\cA_X,\cN) \simeq 0$.
 
Applying the functor $(\cdot) \te_{\cA_X} \fTor_1^{\cB_X}(\cA_X,\cN)$ to the exact sequence
\begin{equation*}
0 \to \hbar \cA_{X}  \to  \cA_X \to \cA_X/ \hbar \cA_X  \to 0, 
\end{equation*}
we get
\begin{equation*}
\hbar \cA_{X} \te_{\cA_X} \fTor_1^{\cB_X}(\cA_X,\cN) \to \fTor_1^{\cB_X}(\cA_X,\cN)\to 0.
\end{equation*}
It follows from Nakayama Lemma (cf. \cite[Lemma 1.2.2]{KS3}) that
\begin{equation*}
\fTor_1^{\cB_X}(\cA_X,\cN)=0.
\end{equation*}

\item Assume that $\cN$ is of $\hbar$-torsion. Since $\cN$ is coherent, there exists $N \in \N^\ast$ such that $\hbar^N \cN=0$. 
Using the exact sequence
\begin{equation*}
0 \to \hbar \cN \to \cN \to \cN / \hbar \cN \to 0
\end{equation*}
it follows by induction on $N$ that it is sufficient to establish that 
\[
\Hn^{-1}(\cA_X \Lte_{\cB_X} \cN)=0
\]
under the assumption that $\hbar^N \cN=0$ with $N=1$.

\item Assume that $\hbar \cN=0$. Then,
\begin{align*}
\cA_X \Lte_{\cB_X} \cN &\simeq  \cA_X \Lte_{\cB_X} (\cB_X/ \hbar \cB_X) \Lte_{\cB_X/ \hbar \cB_X} \cN\\
                         & \simeq \cA_X / \hbar \cA_X \Lte_{\cB_X/ \hbar \cB_X} \cN.
\end{align*}
Since $\cA_X/ \hbar \cA_X$ is flat over $\cB_X/ \hbar \cB_X$, it follows that $\Hn^{-1}(\cA_X \Lte_{\cB_X} \cN)=0$.
\end{enumerate}
\end{proof}

\subsubsection{Polarization for the cotangent bundle of a complex manifold}

Let $M$ be a complex manifold and $X:=T^\ast M$ be the cotangent bundle of $M$ and $\pi:X \to M$ be the projection on the base. An example of polarization on $(X,\chW_X(0))$ is given by the following data:

\begin{enumerate}[(i)]
\item The projection $\pi:X \to M$,
\item The Lagrangian immersion provided by the zero section i.e.  $\iota \colon M \to X$, $p \mapsto (p,0)$,
\item $\cL$ is the quotient of $\chW_X(0)$ by the left ideal $\cI$ of $\chW_X(0)$ generated by $\hbar (\pi^{-1}\Theta_M)$ where $\Theta_M$ is the sheaf of holomorphic vector fields on $M$.
\end{enumerate} 

Identifying $\cL$ and $\cO_M^\hbar$, it follows that if $x$ is a local coordinate system on $M$ and $(x_1,\ldots,x_n,u_1, \ldots,u_n)$ is the associated symplectic coordinate system on $X$ the action of $\chW_X(0)$ on $\cO_M^\hbar$ is given by 
\begin{align*}
x_i \cdot f &=x_i f,\\
u_i \cdot f&= \hbar \partial_{x_i} f
\end{align*}
which agree with the usual quantization rules for spectral curves in the case where $X=T^\ast \C$ (cf. for instance \cite{Norbury}).

\subsubsection{Polarization for $(\C^\ast \times \C^\ast, (dx_1 \wedge d x_2) /(x_1 x_2))$}\label{subsub:cetcet}

We start by constructing a star-algebra $\cA_X$ on $X=(\C^\ast \times \C^\ast, (dx_1 \wedge d x_2) /(x_1 x_2))$ well-suited to study the quantization of the $A$-polynomial. We will also specify a polarization of $(X, \cA_X)$.

Consider the complex surface $X=\C^\ast \times \C^\ast$ with coordinate system $(x_1,x_2)$ and symplectic form $\frac{dx_1 \wedge dx_2}{x_1 x_2}$. We illustrate the use of Proposition \ref{prop:makestaralg}. We define the following sections of $\cD_X^\hbar$.
\begin{equation*}
\begin{array}{lcl}
A_1=x_1 && A_2=x_2 e^{\hbar x_1 \partial_{x_1}}\\
B_1=x_1 e^{\hbar x_2 \partial_{x_2}}                        && B_2=x_2.
\end{array}
\end{equation*}
This sections satisfy condition \eqref{eq:quantcond}. Thus, by Proposition \ref{prop:makestaralg}, we have defined a star-algebra $\cA_X$ on $\C^\ast \times \C^\ast$. This star-algebra is in fact the one given in Example \ref{ex:expexp}.

We now specify the data of a polarization on $(X,\cA_X)$ i.e.

\begin{enumerate}[(i)]
\item the projection $\pi_1: X \to \C^\ast, \; (x_1,x_2) \mapsto x_1$,

\item The Lagrangian immersion given by $\iota \colon \C^\ast \to X, \; x \mapsto (x,1)$,

\item $\cL$ is the quotient of $\cA_X$ by the left ideal $\cI$ of $\cA_X$ generated by the section $x_2 -1$.
\end{enumerate}

Writing $\Lambda$ for the Lagrangian submanifold of $X$ defined by $\lbrace x_{2}=1 \rbrace$ and identifying $\cL$ and $\cO_{\Lambda}^\hbar$, the action of $\cA_X$ on $\cO_{\Lambda}^\hbar$ is given by 
\begin{align*}
x_1 \cdot f &=x_1 f,\\
x_2 \cdot f&= e^{\hbar x_1 \partial_{x_1}} f
\end{align*}
which are the usual quantization rules for coordinate on $\C^\ast \times \C^\ast$ (cf. \cite{Gu,dim} and  \cite{Gulec,Norbury} for surveys).

\subsubsection{Polarizations for $(\C^\ast \times \C, (dx_1 \wedge d x_2) / x_1)$}\label{subsub:ccetoile}

Contrary to the preceding examples in the case of the variety $(\C^\ast \times \C, (dx_1 \wedge d x_2) / x_1)$, there are several natural choices of quantizations and polarizations that correspond respectively to quantum curves encountered in the study of Hurwitz numbers \cite{MulHur} and Gromov-Witten invariants \cite{DNMPS}. We defer the study of the relation between the different quantizations and polarizations to future works.

As in the preceding cases, we start by constructing the star-algebras for which we will specify a polarization.

Consider the complex surface $X=\C^\ast \times \C$ with coordinnate system $(x_1,x_2)$ and symplectic form $\frac{dx_1 \wedge dx_2}{x_1}$. We define the following section of $\cD_X^\hbar$.
\begin{equation*}
\begin{array}{lcl}
A_1=x_1 && A_2=\hbar x_1 \partial_{x_1}+x_2\\
B_1=x_1 e^{\hbar \partial_{x_2}}                        && B_2=x_2.
\end{array}
\end{equation*}
A straightforward computation show that these sections satisfy condition \eqref{eq:quantcond}. Thus, Proposition \ref{prop:makestaralg} ensured the existence of a star-algebra quantizing $X$. The star-product of this star-algebra is the one of Example \ref{ex:affexp} i.e.
\begin{equation*}
f \star g = \sum_{k \geq 0} \left( \dfrac{\hbar^{k}}{k!} \left(\partial_{x_2}^{k} f \right) \left( \sum_{l=0}^k x_1^l \mathcal{S}_k^{(l)} \partial_{x_1}^{l} g \right)  \right).
\end{equation*}

\begin{enumerate}[A)]
\item We now specify a polarization on $(X,\cA_X)$:

\begin{enumerate}[(i)]
\item the projection $\pi_1:\C^\ast \times \C \to \C^\ast$, $(x_1,x_2) \mapsto x_1$

\item The Lagrangian immersion given by $\iota_1: \C^\ast \to X$, $x \mapsto (x,0)$.

\item $\cL_1$ is the quotient of $\cA_X$ by the left ideal $\cI_1$ of $\cA_X$ generated by $x_2$.
\end{enumerate}
Writing $\Lambda_1$ for the Lagrangian submanifold of $X$ defined by $\lbrace x_{2}=0 \rbrace$ and identifying $\cL_1$ and $\cO_{\Lambda_1}^\hbar$, the action of $\cA_X$ on $\cO_{\Lambda_1}^\hbar$ is given by 
\begin{align*}
x_1 \cdot f &=x_1 f,\\
x_2 \cdot f&= \hbar x_1 \partial_{x_1} f
\end{align*}
which are the quantization rules for coordinate on $\C^\ast \times \C$ usually used to study Hurwitz numbers (cf. for instance  \cite{MulHur}).

\item It is also possible to consider the star-algebra $\cA_X^{\opp}$ on $\C^\ast \times \C$. The star-product of this star-algebra is given by
\begin{equation*}
f \circledast g = g \star f= \sum_{k \geq 0} \left( \dfrac{\hbar^{k}}{k!}  \left( \sum_{l=0}^k x_1^l \mathcal{S}_k^{(l)} \partial_{x_1}^{l} f \right)    \left(\partial_{x_2}^{k} g \right) \right).
\end{equation*}

We specify a polarization for $(X,\cA_X^{\opp})$:

\begin{enumerate}[(i)]
\item the projection $\pi_2:\C^\ast \times \C \to \C$, $(x_1,x_2) \mapsto x_2$,

\item The Lagrangian immersion given by $\iota_2: \C \to X$, $x \mapsto (1,x)$,

\item $\cL_2$ is the quotient of $\cA_X^{\opp}$ by the left ideal $\cI_2$ of $\cA_X^{\opp}$ generated by $x_1-1$.
\end{enumerate}

Writing $\Lambda_2$ for the Lagrangian submanifold of $X$ defined by $\lbrace x_{1}-1=0 \rbrace$ and identifying $\cL_2$ and $\cO_{\Lambda_2}^\hbar$, the action of $\cA^{\opp}_X$ on $\cO_{\Lambda_2}^\hbar$ is given by 
\begin{align*}
x_1 \cdot f &= e^{\hbar \partial_{x_2}} f,\\
x_2 \cdot f&= x_2 f
\end{align*}
which are the usual quantization rules for coordinate on $\C^\ast \times \C$ used in the study of Gromov-Witten invariants (cf. for instance \cite{DNMPS}).
\end{enumerate}

\subsection{Some algebras of operators}
Generally, quantum curves are sections of subalgebras of $\cD_\Lambda^\hbar$. This is especially the case for quantum curves quantizing spectral curves associated to Higgs bundles or quantizing the $A$-polynomials or related to Gromow-Witten theory. That is, in situation where the quantum curves are arising from topological recursion. In this subsection, we construct such algebras and study their properties.

The case of quantum curves related Hurwitz numbers is different. Two types of quantum curves should be distinguished those involving derivative in the direction of the deformation parameter $\hbar$ and those who do not. It seems that quantum curves of the second type define section of the star-algebra defined in subsection \ref{subsub:ccetoile} B) whereas quantum curves of the first type correspond to section of quantization algebras (see \cite{AK}). Quantization algebras and DQ-algebras are closely related. We defer to future work the detailed study, from the point of view of DQ-algebras and quantization algebras, of quantum curves appearing in the study of Hurwitz numbers.

\subsubsection{Rees $\cD$-modules}

Several authors have studied the quantization of spectral curves associated to Higgs bundles in terms of modules over the Rees-algebra of differential operators filtered by the order (See for instance \cite{Dij,Mu2,Mul}). We call these modules Rees $\cD$-modules. In this section, we study the relation between DQ-modules and Rees $\cD$-modules.

Let $M$ be a complex manifold and $X:=T^\ast M$ its cotangent bundle. On $M$ we consider the sheaf $\cD_M$ of holomorphic differential operators. For every $j \in \N$, we write $\cD_M(j)$ for the $j^{th}$ piece of the filtration by the order of $\cD_M$ and $\cD_M[\hbar]$ for $\cD_M \te_\C \C[\hbar]$ and $\cD_M[\hbar,\hbar^{-1}]$ for $\cD_M \te_\C \C[\hbar, \hbar^{-1}]$.

\begin{definition}
The Rees algebra of $\cD_M$ is the subsheaf of $\C[\hbar] \te_\C \cD_M$
\begin{equation*}
R(\cD_M)= \bigoplus_{j=0}^{\infty} \hbar^j \cD_M(j).
\end{equation*}
\end{definition}

It follows from the definition of $R(\cD_M)$ that it is an algebra over $\C[\hbar]_M$ and that the inclusion provides a morphism of algebras 
\begin{equation}\label{map:moralg}
R(\cD_M) \hookrightarrow \cD_M[\hbar].
\end{equation}

\begin{lemma}\label{lem:localisation} There is an isomorphism of algebras 
\begin{equation}\label{map:localgrees}
 R(\cD_M) \te_\C \C[\hbar^{-1}]\stackrel{\sim}{\to} \cD_M[\hbar,\hbar^{-1}].
\end{equation}
\end{lemma}

\begin{proof}
Tensoring the morphism \eqref{map:moralg} by $ (\cdot) \te_ \C \C[\hbar^{-1}]$, we obtain the morphism $R(\cD_M) \te_\C \C[\hbar^{-1}]\to \cD_M[\hbar,\hbar^{-1}]$ which is clearly a monomorphism. Locally, one checks that $R(\cD_M) \otimes \C[\hbar^{-1}]$ contains $\cO_M$ and the sheaf of vector fields $\Theta_M$ which implies that the morphism \eqref{map:localgrees} is an epimorphism. 
\end{proof}
We have the following result.
\begin{proposition}[{\cite[Theorem A.34]{KDmod}}] \label{prop:cohRD}
The $\cO_M$-algebra $R(\cD_M)$ is Noetherian.
\end{proposition}

We endow $X=T^\ast M$ with the DQ-algebra $\chW_X(0)$. Recall that we have a flat morphism of algebras
\begin{equation*}
\pi^{-1} \cD_M \hookrightarrow \chW_X
\end{equation*}
such that if $(x_1,\ldots,x_n;u_1,\ldots,u_n)$ is a local symplectic coordinate system on $X$, then $x_i \mapsto x_i$ and $\partial_{x_i} \mapsto \hbar^{-1}u_i$. This induces a morphism of algebras
\begin{equation*}
\pi^{-1} \cD_M[\hbar, \hbar^{-1}] \hookrightarrow \chW_X.
\end{equation*}
By composing the above morphism with morphism \eqref{map:moralg} we get a map
\begin{equation*}
\Psi: \pi^{-1} R(\cD_M) \hookrightarrow \chW_X.
\end{equation*}
It is clear that $\Psi(\pi^{-1}R(\cD_M)) \subset \chW_X(0)$. 

We summarize the situation in the following commutative diagram of morphisms of algebras.
\begin{equation*}
\xymatrix{\pi^{-1} \cD_M[\hbar, \hbar^{-1}] \ar@{^{(}->}[r] & \chW_X\\
                    \pi^{-1}R(\cD_M) \ar@{^{(}->}[u]       \ar@{^{(}->}[r]  & \chW_X(0). \ar@{^{(}->}[u]
}
\end{equation*}

We endow $\chW_X(0)$ with the canonical filtration defined in \eqref{filt:DQalg} i.e.

\begin{equation*}
\cW_X(k)=
\begin{cases}
\hbar^{-k}\cW_X(0) &\mbox{if } k < 0\\
\cW_X(0)           &\mbox{if } k \geq 0.\\
\end{cases}
\end{equation*}

This filtration induces a filtration $(R(\cD_M)(k))_{k \in \Z}$ on $\pi^{-1}R(\cD_M)$ such that $\hbar$ is in degree $-1$ in $\pi^{-1}R(\cD_M)$ and $\hbar \partial_{x_i}$ is in degree zero in $\pi^{-1}R(\cD_M)$.

We denote by $\cO_{[X]}$ the sub-ring of $\cO_X$ the sections of which are polynomial in the fibers. We have 
\begin{equation*}
\chW_X(0) / \chW_X(-1) \simeq \chW_X(0) / \hbar \chW_X(0) \simeq \cO_X
\end{equation*}
and 
\begin{equation*}
R(\cD_M)(0) / R(\cD_M)(-1) \simeq \cO_{[X]}.
\end{equation*}
We recall the following well-known fact. 
\begin{proposition}\label{prop:flat}
The sheaf of $\C$-algebras $\cO_X$ is flat over the sheaf of algebras $\cO_{[X]}$.
\end{proposition}

\begin{proposition}\label{prop:DQflatDH} 
\begin{enumerate}[(i)]
\item The ring $\chW_X(0)$ is flat over $\pi^{-1}R(\cD_M)$,
\item The algebra $\chW_X$ is flat over $\pi^{-1}\cD_M[\hbar, \hbar^{-1}]$,
\item The algebra $\chW_X$ is flat over $\pi^{-1}R(\cD_M)$.
\end{enumerate}
\end{proposition}

\begin{proof} 
\begin{enumerate}[(i)]
\item It follows from Proposition \ref{prop:cohRD} and Proposition \ref{prop:flat} that the hypothesis of Theorem \ref{thm:flatcrit} are satisfied which proves the claim.
\item $\chW_X$ is the localization of $\chW_X(0)$ with respect to $\hbar$ and $\pi^{-1}\cD_M[\hbar, \hbar^{-1}]$ is the localization of $\pi^{-1}R(\cD_M)$ with respect to $\hbar$. As $\chW_X(0)$ is flat over $\pi^{-1}R(\cD_M)$ by (i), the result follows.
\item This follows from (ii) and from the fact that $\pi^{-1}\cD_M[\hbar, \hbar^{-1}]$ is flat over $R(\cD_M)$.
\end{enumerate}
\end{proof}

We set
\begin{align*}
(\cdot)^{W_0}:&\Mod(R(\cD_M)) \to \Mod(\chW_X(0)) \\ 
&\cM \mapsto \chW_X(0)\te_{\pi^{-1}R(\cD_M)} \pi^{-1}\cM.
\end{align*}

\begin{corollary}\label{cor:SolRD}
Let $\cM \in \Mod_{\coh}(R(\cD_M))$. Then,
\begin{equation*}
\fRHom_{R(\cD_M)}(\cM, \cO^\hbar_{M}) \simeq \mathrm{R} \pi_\ast  \fRHom_{\chW_X(0)}(\cM^{W_0}, \cO^\hbar_{M}).
\end{equation*}
\end{corollary}

\begin{proof}
This follows from Proposition \ref{prop:compsol}.
\end{proof}

\subsubsection{Scaling operators} According to \cite{Gu}, the quantum curves appearing in the quantization of the $A$-polynomial are scaling operators. In this sub-section, we introduce the algebra formed by such operators, study its properties and relate it to DQ-algebras.

Consider $\C^\ast$ with the coordinate $x$. Let $\cS_{\C^\ast}$ be the sub-algbera of $\cD_{\C^\ast}^\hbar$ generated by $\cO_{\C^\ast}^\hbar$, $e^{\hbar x \partial_x}$ and $e^{-\hbar x \partial_x}$. Let $\cS^{+}_{\C^\ast}$ be the sub-algbera of $\cD_{\C^\ast}^\hbar$ generated by $\cO_{\C^\ast}^\hbar$ and $e^{\hbar x \partial_x}$. We write $S$ for the operator $e^{\hbar x \partial_x}$ and since $e^{\hbar x \partial_x} e^{-\hbar x \partial_x}=e^{-\hbar x \partial_x}e^{\hbar x \partial_x}=\id$, the operator $e^{-\hbar x \partial_x}$ is naturally denoted by $S^{-1}$. Finally, we denote by $\theta_S:\cS^{+}_{\C^\ast} \to \cS_{\C^\ast}$ the inclusion of $\cS^{+}_{\C^\ast}$ into $\cS_{\C^\ast}$.

\begin{lemma} \label{lem:relation}
Let $f \in \cO_{\C^\ast}^\hbar$, then 
\begin{enumerate}[(i)]
\item $S \circ f=S(f) \circ S$,
\item $S^{-1} \circ f=S^{-1}(f) \circ S^{-1}$.
\end{enumerate}
\end{lemma}

\begin{proof}
\noindent$(i)$ If $f=\sum_{n \geq 0} f_n \hbar^n$ then, $S(f)=\sum_{n \geq 0} S(f_n) \hbar^n$. Thus, we just need to prove the result for $f \in \cO_{\C^\ast}$. Then, using the local biholomorphism given by the change of variable $x=e^u$ and writing $F(u)=f(e^u)$, we have
\begin{align*}
S \circ f&= \sum_{k \geq 0} \frac{\hbar^k}{k!} (x \partial_x)^k \circ f
         = \sum_{k \geq 0} \frac{\hbar^k}{k!} \partial_u^k \circ F\\
         &=\sum_{k \geq 0} \frac{\hbar^k}{k!} \sum_{p=0}^k \binom{k}{p} \partial_u^{k-p} F \partial_u^{k-p}=\sum_{k \geq 0} \hbar^k \sum_{p=0}^k \frac{1}{(k-p)!p!} \partial_u^{k-p} F \partial_u^p\\ 
         &=S(f) \circ S.
\end{align*}
\noindent $(ii)$ Applying formula $(i)$ to $S^{-1}(f)$, we get that $S \circ S^{-1}(f)= f \circ S$. This implies that $S^{-1}(f) \circ S^{-1}=S^{-1} \circ f$.
\end{proof}

\begin{lemma}\label{lem:orelocS}
Let $x \in \C^\ast$, then $\theta_{Sx} \colon \cS^+_{\C^\ast, x} \to \cS_{\C^\ast, x}$ is the Ore localisation of $\cS^+_{\C^\ast, x}$ with respects to the multiplicative set $\lbrace e^{n \hbar x \partial_x} \rbrace_{n \in \N}$.
\end{lemma}

\begin{proof}
This follows immediately from Lemma \ref{lem:relation}.
\end{proof}

\begin{proposition}\label{prop:ecriture}
Let $P \in \cS_{\C^\ast}$. Locally $P$ can be written in a unique way in the form
\begin{equation*}
P=\sum_{k=-m}^n f_k S^k
\end{equation*}
where $f_k \in \cO_{\C^\ast}^\hbar$ for $-m \leq k \leq n$ and $m,n \in \N$. 
\end{proposition}

\begin{proof}
The existence follows from Lemma \ref{lem:relation}. We now prove the uniqueness. For that purpose it is sufficient to prove that for every $m,n \in \N$, 
\begin{equation}\label{eq:unicite}
\sum_{k=-m}^n f_k S^k=0
\end{equation}
implies that for every $-m \leq k \leq n$, $f_k=0$. 

By considering the composition $P \circ S^m$ we can assume that $m=0$. Let $(f_k)_{0 \leq k \leq n}$ such that Equation \eqref{eq:unicite} holds. Then,
\begin{align*}
\sum_{k=0}^n f_k S^k (x^p)&=\sum_{k=0}^n f_k e^{pk\hbar} x^p=0.\\
\end{align*}
Evaluating in $x_0 \in \C^\ast$, we have

\begin{equation*}
\sum_{k=0}^n f_k(x_0) e^{pk\hbar}=0.
\end{equation*}
Then, the polynom $f(z)=\sum_{k=0}^n f_k(x_0) z^k \in \C^{\hbar,loc}[z]$ has infinitely many roots (the $e^{p \hbar} \in \C^{\hbar,loc}$ for $ p\in \N$). It follows that $f(z)=0$ which implies that for every $x_0 \in \C^\ast$, $f_k(x_0)=0$ which proves the claim.
\end{proof}
\begin{remark}
In view of Lemma \ref{lem:relation} (i), it is possible to write locally any $P \in \cS_{\C^\ast}$ in the form
\[
P=\sum_{k=-m}^n S^k \circ g_k
\]
where $g_k \in \cO_{\C^\ast}^\hbar$ for $-m \leq k \leq n$ and $m,n \in \N$. 
\end{remark}
We will need the following finiteness result.
\begin{proposition}\label{prop:Spcoh}
The algebra $\cS^{+}_{\C^\ast}$ is a Noetherian $\cO_{\C^\ast}^\hbar$-algebra.
\end{proposition}

\begin{proof}
It follows from point $(i)$ of Lemma \ref{lem:relation} and the fact that $S$ is an automorphism of $\cO_{\C^\ast}$ that the hypothesis  of \cite[Theorem 5.1.1]{AK} are statisfied. This implies that $\cS^{+}_{\C^\ast}$ is Noetherian.
\end{proof}

\begin{corollary}\label{cor:cohS}
The $\cO^\hbar_{\C^\ast}$-algebra $\cS_{\C^\ast}$ is coherent.
\end{corollary}

\begin{proof}
The left $\cS_{\C^\ast}$-module $\cS_{\C^\ast}$ is clearly finitely generated over itself.
Let $U \subset X$ be an open set and let $x \in U$ and let $f: \cS^m_{\C^\ast}|_U \to \cS_{\C^\ast}|_U$ be a morphism of $\cS_{\C^\ast}$-modules.
Let $(E_i)_{1 \leq i \leq m}$ be the canonical basis of $\cS^m_{\C^\ast}$. We set $Q_i=f(E_i)$.

It follows from Proposition \ref{prop:ecriture} that there exists an open set $V \subset U$ containing $x$ and $m \in \N$ such that 
\[
Q_i=(\sum_{k=0}^{n_i} f_{ik} S^k)S^{-m}. 
\]
Then, we get a morphism of left $\cS_{\C^\ast}^+|_V$-modules 
\[
\tilde{f}:\cS_{\C^\ast}^{+m}|_V \to \cS_{\C^\ast}^+|_V, \; P \mapsto f(P)S^m. 
\]
Since $\cS^{+}_{\C^\ast}$ is a coherent sheaf, it follows that $\ker(\tilde{f})$ is a locally finitely generated $\cS^{+}_{\C^\ast}|_V$-module. 

Since $f|_V \circ [\theta_S,\ldots, \theta_S]=\tilde{f} S^{-m}$ then, $\ker\tilde{f} \simeq \ker (f|_V \circ [\theta_S,\ldots,\theta_S])$ and $\cS_{\C^\ast} (\ker (f|_V \circ [\theta_S,\ldots,\theta_S])) \simeq \ker (f|_V)$ which proves that $\ker (f|_V)$ is finitely generated. This proves that $\cS_{\C^\ast}$ is coherent.
\end{proof}

We set $X= \C^\ast \times \C^\ast$ and denote by $\cA_X$ the star-algebra defined in the Subsection \ref{subsub:cetcet} and consider the polarization on $(X, \cA_X)$ defined in the subsection just mentionned. We relate the algebras $\cS_{\C^\ast}$ and $\cA_X$.

It follows from Proposition \ref{prop:ecriture} that there is a morphism of left $\pi^{-1} \cO_{\C^\ast}$-modules defined by
\begin{equation}\label{mor:AQ}
\phi \colon \pi^{-1} \cS_{\C^\ast} \to \cA_X, \; f(x) \ni  \cO_{\C^\ast}^\hbar \mapsto f(x_1), \; S^n \mapsto x_2^n \; \textnormal{for $n \in \Z$}.
\end{equation}

\begin{proposition}\label{prop:morofalg}
The morphism \eqref{mor:AQ} is a morphism of $\pi^{-1} \cO^\hbar_{\C^\ast}$-algebras.
\end{proposition}

\begin{proof}
By definition $\phi$ is a morphism of $\pi^{-1} \cO_{\C^\hbar}$-module and it is clear that 
\begin{equation*}
\phi(S^2)=x_2^2=x_2 \star x_2.
\end{equation*} 
Thus, the only thing we need to check is that 
\begin{equation*}
\phi(Sf)=\phi(S) \phi(f).
\end{equation*}
But, 
\begin{align*}
\phi(Sf)&=\phi(S(f)S)=S(f)\star x_2\\
        &=x_2 \star f=\phi(S) \star \phi(f).
\end{align*}
\end{proof}

We endow the algebra $\cA_X$ with the filtration defined in \eqref{filt:DQalg}.

This filtration induces a filtration $(\cS_{\C^\ast}(k))_{k \in \Z}$ on $\pi^{-1}\cS_{\C^\ast}$ such that for every $(p_1,p_2) \in X$, we have the following isomorphism
\begin{equation*}
\cS_{\C^\ast}(0)_{(p_1,p_2)} / \cS_{\C^\ast}(-1)_{(p_1,p_2)} \simeq \cO_{\C^\ast,\;p_1}[x_2,x_2^{-1}].
\end{equation*}

\begin{lemma}\label{lem:flatquotientS}
The sheaf of rings $\cA_X(0) / \cA_X(-1)$ is flat over the sheaf of rings $\cS_{\C^\ast}(0) / \cS_{\C^\ast}(-1)$.
\end{lemma}

\begin{proof}
We need to prove that for any point $(p_1,p_2) \in \C^\ast \times \C^\ast$, 

\[
\cA_X(0)_{(p_1,p_2)} / \cA_X(-1)_{(p_1,p_2)} \simeq \cO_{\C^\ast \times \C^\ast,\; (p_1,p_2)}
\] is flat over 
\[
\cS_{\C^\ast}(0)_{(p_1,p_2)} / \cS_{\C^\ast}(-1)_{(p_1,p_2)} \simeq \cO_{\C^\ast,\;p_1}[x_2,x_2^{-1}]. 
\]
The morphism of algebras 
\begin{equation*}
\cO_{\C^\ast,\;p_1}[x_2] \hookrightarrow \cO_{\C^\ast \times \C^\ast,\;(p_1,p_2)}  
\end{equation*}
is flat. Moreover, $\cO_{\C^\ast,\;p_1}[x_2,x_2^{-1}]$ is the localisation of $\cO_{\C^\ast,\;p_1}[x_2]$ with respect to $x_2$ which is already invertible in $\cO_{\C^\ast \times \C^\ast, \; (p_1,p_2)}$. Thus, the ring $\cO_{\C^\ast \times \C^\ast,\;(p_1,p_2)}$ is flat over $\cO_{\C^\ast,\;p_1}[x_2]$ which proves the claim.
\end{proof}

\begin{proposition}\label{prop:flatS} \begin{enumerate}[(i)]
\item The ring $\cA_X$ is flat over the ring $\pi^{-1} \cS_{\C^\ast}$,
\item The ring $\cS_{\C^\ast}$ is flat over $\cS^+_{\C^\ast}$,
\item  The ring $\cA_X$ is flat over the ring $\pi^{-1}\cS^+_{\C^\ast}$.
\end{enumerate} 
\end{proposition}

\begin{proof}
\begin{enumerate}[(i)]
\item Corollary \ref{cor:cohS} implies that $\pi^{-1}\cS_{\C^\ast}$ is coherent and Lemma \ref{lem:flatquotientS} states that $\cA_X(0) / \cA_X(-1)$ is flat over $\cS_{\C^\ast}(0) / \cS_{\C^\ast}(-1)$. Thus, the claim follows from Theorem \ref{thm:flatcrit}.
\item This follows from Lemma \ref{lem:orelocS} and Proposition 2.1.16 of \cite{Rob}.
\item This is a consequence of (i) and (ii).
\end{enumerate}
\end{proof}

Consider the functor
\begin{align*}
(\cdot)^{\cA_X} \colon \Mod_\coh(\cS_{\C^\ast}) &\to  \Mod_\coh(\cA_X) \\
\cN & \mapsto \cA_X \te_{\pi^{-1}\cS_{\C^\ast}} \pi^{-1} \cN.
\end{align*}

\begin{corollary}\label{cor:SolS}
Let $\cM \in \Mod_{\coh}(\cS_{\C^\ast})$. Then,
\begin{align*}
\fRHom_{\cS_{\C^\ast}}(\cM, \cO^\hbar_{\C^\ast}) \simeq& \mathrm{R} \pi_\ast  \fRHom_{\cA_X}(\cM^{\cA_X}, \cO^\hbar_{\C^\ast}).
\end{align*}
\end{corollary}

\begin{proof}
This follows from Proposition \ref{prop:compsol}.
\end{proof}

\subsubsection{Translation operators} In view of \cite{DNMPS}, it seems that the quantum curves appearing in the study of Gromow-Witten invariants are translation operators. In this sub-section we introduce the algebra formed by such operators and relate it to the theory of DQ-modules.

Consider $\C$ with the coordinate $x$. Let $\cT_{\C}$ be the sub-algebra of $\cD_{\C}^\hbar$ generated by $\cO_{\C}$ and $e^{\hbar \partial_x}$ and $e^{-\hbar \partial_x}$. Let $\cT^{+}_{\C^\ast}$ be the sub-algbera of $\cD_{\C}^\hbar$ generated by $\cO_{\C}^\hbar$ and $e^{\hbar \partial_x}$. We denote by $T$ the operator $e^{\hbar \partial_x}$ and since $e^{-\hbar \partial_x}$ is the inverse of $T$, we naturally denote it by $T^{-1}$. Finally, we denote by $\theta_T:\cT^{+}_{\C} \to \cT_{\C}$ the inclusion of $\cT^{+}_{\C}$ into $\cT_{\C}$. The proofs of the different results of this sub-section are very similar to the proofs of the previous section. Thus, we do not repeat certain arguments.

\begin{lemma}\label{lem:relationT}
Let $f \in \cO_{\C}^\hbar$.
\begin{enumerate}[(i)]
\item $T \circ f=T(f) \circ T$,
\item $T^{-1} \circ f=T^{-1}(f) \circ T^{-1}$.
\end{enumerate}
\end{lemma}

\begin{proof}
The proof is similar to the proof of Lemma \ref{lem:relation}.
\end{proof}

\begin{lemma}\label{lem:orelocT}
Let $x \in \C$, then $\theta_{Tx} \colon \cT^+_{\C, x} \to \cT_{\C, x}$ is the Ore localisation of $\cT^+_{\C, x}$ with respects to the multiplicative set $\lbrace e^{n \hbar  \partial_x} \rbrace_{n \in \N}$.
\end{lemma}

\begin{proof}
This follows immediately from Lemma \ref{lem:relationT}.
\end{proof}

\begin{proposition}\label{prop:transecun}
Let $P \in \cT_{\C}$. Locally $P$ can be written in a unique way in the form
\begin{equation*}
P=\sum_{k=-m}^n f_k T^k
\end{equation*}
where $f_k \in \cO_{\C}^\hbar$ for $-m \leq k \leq n$ and $m,n \in \N$. 
\end{proposition}

\begin{proof}
The existence of such a form follows from Lemma \ref{lem:relationT}. To prove the uniqueness, it is sufficient to show that $\sum_{k=-m}^nf_k T^k=0$ implies that for every $-m \leq k \leq n, \; f_k=0$.
By considering the composition $P \circ T^m$ we can assume that $m=0$. We evaluate $P$ on $x^p$ and obtain
\begin{align*}
P(x^p)&=\sum_{k=0}^n f_k(x) T^k(x)\\
      &=\sum_{k=0}^n f_k(x) (x+h)^{kp}.
\end{align*}
For every $x_0 \in \C$ and $p \in \N$, the elements $(x_0+\hbar)^p$ are roots of the polynomial $\sum_{k=0}^n f_k(x_0)z^k \in \C^{\hbar,loc}[z]$. Thus, $f_k(x_0)=0$.
\end{proof}

\begin{proposition}
The algebra $\cT^+_\C$ is a Noetherian $\cO_\C^\hbar$-algebra.
\end{proposition}

\begin{proof}
It follows from point $(i)$ of Lemma \ref{lem:relationT} and the fact that $T$ is an automorphism of $\cO_{\C}$ that the hypothesis  of \cite[Theorem 5.1.1]{AK} are satisfied. This implies that $\cT^{+}_{\C}$ is Noetherian.
\end{proof}

\begin{corollary}
The algebra $\cT_\C$ is a coherent $\cO^\hbar_\C$-algebra.
\end{corollary}

\begin{proof}
The proof in analogous to the proof of Corollary \ref{cor:cohS}.
\end{proof}

We let $X$ be the symplectic surface $(\C^\ast \times \C,(dx_1 \wedge dx_2)/x_1)$ and consider the star-algebra $\cA_X$ (we write $\cA_X$ instead of $\cA_X^{\opp}$) and the polarization on $(X, \cA_X)$ which are defined in subsection \ref{subsub:ccetoile} B). It follows from Proposition \ref{prop:transecun} that there is a morphism of $\pi^{-1}\cO^\hbar_{\C}$-modules defined by
\begin{equation}\label{mor:AT}
\psi \colon \pi^{-1} \cT_{\C} \to \cA_X, \; f(x) \ni  \cO_{\C}^\hbar \mapsto f(x_2), \; T^n \mapsto x_1^n \; \textnormal{for $n \in \Z$}.
\end{equation}

\begin{proposition}
The morphism \eqref{mor:AT} is a morphism of $\pi^{-1}\cO_{\C}^\hbar$-algebras.
\end{proposition}

\begin{proof}
The proof is similar to the proof of Proposition \ref{prop:morofalg}.
\end{proof}

We endow the algebra $\cA_X$ with the filtration defined in \eqref{filt:DQalg}. This filtration induces a filtration $(\cT_{\C}(k))_{k \in \Z}$ on $\pi^{-1}\cT_{\C}$ such that for every $(p_1,p_2) \in X$, we have the following isomorphism
\begin{equation*}
\cT_{\C}(0)_{(p_1,p_2)} / \cT_{\C}(-1)_{(p_1,p_2)} \simeq \cO_{\C,\;p_1}[x_1,x_1^{-1}].
\end{equation*}

\begin{lemma}\label{lem:flatquotientT}
The sheaf of rings $\cA_X(0) / \cA_X(-1)$ is flat over the sheaf of rings $\cT_{\C}(0) / \cT_{\C}(-1)$.
\end{lemma}

\begin{proof}
The proof goes as the proof of Lemma \ref{lem:flatquotientS}.
\end{proof}

\begin{proposition}\label{prop:flatT} \begin{enumerate}[(i)]
\item The ring $\cA_X$ is flat over the ring $\pi^{-1} \cT_{\C}$,
\item The ring $\cT_{\C}$ is flat over $\cT^+_{\C}$,
\item  The ring $\cA_X$ is flat over the ring $\pi^{-1}\cT^+_{\C}$.
\end{enumerate} 
\end{proposition}

\begin{proof}
The proof goes exactly as the proof of Proposition \ref{prop:DQflatDH}.
\end{proof}

Consider the functor
\begin{align*}
(\cdot)^{\cA_X} \colon \Mod_\coh(\cT_{\C}) &\to  \Mod_\coh(\cA_X) \\
\cN & \mapsto \cA_X \te_{\pi^{-1}\cT_{\C}} \pi^{-1} \cN.
\end{align*}

We have the following result
\begin{corollary}\label{cor:SolT}
Let $\cM \in \Mod_{\coh}(\cT_{\C})$. Then,
\begin{align*}
\fRHom_{\cT_{\C}}(\cM, \cO^\hbar_{\C}) \simeq& \mathrm{R} \pi_\ast  \fRHom_{\cA_X}(\cM^{\cA_X}, \cO^\hbar_{\C}).
\end{align*}
\end{corollary}
\begin{proof}
This follows from Proposition \ref{prop:compsol}.
\end{proof}
% Non-BibTeX users please use


\begin{thebibliography}
{}\bibitem{GinPech}
V.~Baranovsky, V.~Ginzburg, D.~Kaledin, and J.~Pecharich.
\newblock Quantization of line bundles on Lagrangian subvarieties.
\newblock {\em Selecta Mathematica}, pages 1--25, 2015.

\bibitem{BeS1}
H.~Behnke and K.~Stein.
\newblock Entwicklung analytischer funktionen auf riemannschen fl{\"a}chen.
\newblock {\em Math. Ann.}, (120):430--461, 1948.

\bibitem{BeS2}
H.~Behnke and K.~Stein.
\newblock Elementarfunktionen auf riemannschen fl{\"a}chen als hilfsmittel
  f{\"u}r die funktionentheorie mehrerer ver{\"a}nderlichen.
\newblock {\em Can. J. Math. 2}, (2):152--165, 1950.

\bibitem{BorA}
G.~Borot and B.~Eynard.
\newblock All order asymptotics of hyperbolic knot invariants from
  non-perturbative topological recursion of a-polynomials.
\newblock {\em Quantum Topology}, 6:40--138, 2015.

\bibitem{AK}
A.~{D'}Agnolo and M.~Kashiwara.
\newblock On quantization of complex symplectic manifolds.
\newblock {\em Communications in Mathematical Physics}, 308(1):81--113, 2011.

\bibitem{DS}
A.~D'Agnolo and P.~Schapira.
\newblock Quantization of complex {L}agrangian submanifolds.
\newblock {\em Advances in Mathematics}, 213(1):358--379, 2007.

\bibitem{Dij}
R.~Dijkgraaf, L~Hollands, and P.~Su{\l}kowski.
\newblock Quantum curves and {$\mathcal{D}$}-modules.
\newblock {\em Journal of High Energy Physics}, 2009(11):047, 2009.

\bibitem{dim}
T.~Dimofte.
\newblock Quantum riemann surfaces in chern-simons theory.
\newblock {\em Adv. Theor. Math. Phys.}, 17(3):479--599, 06 2013.

\bibitem{DoM}
N.~Do and D.~Manescu.
\newblock Quantum curves for the enumeration of ribbon graphs and hypermaps.
\newblock {\em Commun.Num.Theor.Phys.}, pages 677--701, 08 2014.

\bibitem{Mu2}
O.~{Dumitrescu} and M.~{Mulase}.
\newblock {Quantization of spectral curves for meromorphic Higgs bundles
  through topological recursion}.
\newblock {\em ArXiv:1411.1023}.

\bibitem{Mul}
O.~Dumitrescu and M.~Mulase.
\newblock Quantum curves for {H}itchin fibrations and the {E}ynard-{O}rantin
  theory.
\newblock {\em Letters in Mathematical Physics}, 104(6):635--671, 2014.

\bibitem{DNMPS}
P.~Dunin-Barkowski, M.~Mulase, P.~Norbury, A.~Popolitov, and S.~Shadrin.
\newblock Quantum spectral curve for the {G}romov-{W}itten theory of the
  complex projective line.
\newblock {\em Journal f\"ur die reine und angewandte Mathematik (Crelles
  Journal)}, 2014.

\bibitem{EO}
B.~Eynard and N.~Orantin.
\newblock Invariants of algebraic curves and topological expansion.
\newblock {\em Communications in Number Theory and Physics}, 1:40--138, 20107.

\bibitem{Forstneric}
F.~Forstneri\v{c}.
\newblock {\em Stein Manifolds and Holomorphic Mappings: The Homotopy Principle
  in Complex Analysis}, volume~56 of {\em Ergebnisse der Mathematik und ihrer
  Grenzgebiete. 3. Folge / A Series of Modern Surveys in Mathematics}.
\newblock Springer-Verlag, Berlin, 2011.

\bibitem{EGAIII1}
A.~{Grothendieck}.
\newblock {\'El\'ements de g\'eom\'etrie alg\'ebrique. I: Le langage des
  sch\'emas. II: \'Etude globale \'el\'ementaire de quelques classe de
  morphismes. III: \'Etude cohomologique des faisceaux coh\'erents (premi\`ere
  partie).}
\newblock {\em {Publ. Math., Inst. Hautes \'Etud. Sci.}}, 4:1--228, 1960.

\bibitem{Guche}
S.~Gukov.
\newblock Three-{D}imensional {Q}uantum {G}ravity, {C}hern-{S}imons {T}heory,
  and the {A}-polynomial.
\newblock {\em Communications in Mathematical Physics}, 255(3):577--627, 2005.

\bibitem{Gulec}
S.~{Gukov} and I.~{Saberi}.
\newblock {Lectures on Knot Homology and Quantum Curves}.
\newblock {\em arXiv:1211.6075}, November 2012.

\bibitem{Gu}
S.~Gukov and P.~Sulkowski.
\newblock {A}-polynomial, {B}-model, and quantization.
\newblock {\em Journal of High Energy Physics}, 2012(2):1--57, 2012.

\bibitem{KDmod}
M.~Kashiwara.
\newblock {\em {$D$}-modules and microlocal calculus}, volume 217 of {\em
  Translations of Mathematical Monographs}.
\newblock American Mathematical Society, Providence, RI, 2003.

\bibitem{KS1}
M.~Kashiwara and P.~Schapira.
\newblock {\em Sheaves on manifolds}, volume 292 of {\em Grundlehren der
  Mathematischen Wissenschaften [Fundamental Principles of Mathematical
  Sciences]}.
\newblock Springer-Verlag, Berlin, 1990.

\bibitem{KS3}
M.~Kashiwara and P.~Schapira.
\newblock {\em Deformation quantization modules}, volume 345 of {\em
  Ast\'erisque}.
\newblock Soc. Math. France, 2012.

\bibitem{Kos}
M.~Kontsevich.
\newblock Deformation quantization of algebraic varieties.
\newblock {\em Lett. Math. Phys.}, 56(3):271--294, 2001.
\newblock EuroConf{\'e}rence Mosh{\'e} Flato 2000, Part III (Dijon).

\bibitem{Rob}
J.~C. McConnell and J.~C. Robson.
\newblock {\em Noncommutative {N}oetherian rings}.
\newblock Pure and Applied Mathematics (New York). John Wiley \& Sons, Ltd.,
  Chichester, 1987.

\bibitem{MulHur}
M.~Mulase, S.~Shadrin, and L.~Spitz.
\newblock {The spectral curve and the Schr\"odinger equation of double Hurwitz
  numbers and higher spin structures}.
\newblock {\em Commun.Num.Theor Phys.}, 07:125--143, 2013.

\bibitem{Norbury}
P.~{Norbury}.
\newblock {Quantum curves and topological recursion}.
\newblock {\em ArXiv:1502.04394}, 2015.

\bibitem{PolSch}
P.~Polesello and P.~Schapira.
\newblock Stacks of quantization-deformation modules on complex symplectic
  manifolds.
\newblock {\em International Mathematics Research Notices},
  2004(49):2637--2664, 2004.

\bibitem{SKK}
M.~Sato, T.~Kawai, and M.~Kashiwara.
\newblock Microfunctions and pseudo-differential equations.
\newblock In {\em Hyperfunctions and pseudo-differential equations ({P}roc.
  {C}onf., {K}atata, 1971; dedicated to the memory of {A}ndr\'e {M}artineau)},
  pages 265--529. Lecture Notes in Math., Vol. 287. Springer, Berlin, 1973.

\bibitem{Sch}
P.~Schapira.
\newblock {\em Microdifferential systems in the complex domain}, volume 269 of
  {\em Grundlehren der Mathematischen Wissenschaften}.
\newblock Springer-Verlag, Berlin, 1985.
\end{thebibliography}
\end{document}